\documentclass[twoside,11pt]{article}

\usepackage{jmlr2e}
\usepackage{amsmath,amssymb}
\usepackage{bm}
\usepackage{graphicx,float,wrapfig}
\usepackage{subcaption}
\usepackage{wrapfig}
\usepackage{ifpdf}
\usepackage{listings}
\usepackage{diagbox}
\usepackage{algorithm}
\usepackage{algorithmicx}
\usepackage[noend]{algpseudocode}
\usepackage{enumitem}
\usepackage{multirow}
\usepackage{multicol}
\usepackage{color}
\usepackage{soul}
\usepackage{hhline}
\usepackage{thm-restate}
\usepackage{graphicx}
\usepackage[font=small,labelfont=bf]{caption}
\usepackage{mathtools}
\usepackage{tikz}
\usetikzlibrary{arrows.meta,calc,decorations.markings,math,arrows.meta}
\tikzset{every picture/.style={line width=1pt}} %

\newcommand{\expect}{\mathbb{E}}

\newcommand{\real}{\mathbb{R}}

\newcommand{\bx}{x}
\newcommand{\by}{y}

\newcommand{\bigo}{\mathcal{O}}

\newcommand{\bz}{z}

\newcommand{\bv}{v}

\newcommand{\zero}{\mathbf{0}}

\newcommand{\iden}{\mathbf{I}}
\newcommand{\bmat}[1]{\begin{bmatrix}#1\end{bmatrix}}
\newcommand{\bsmat}[1]{\left[ \begin{smallmatrix} #1 \end{smallmatrix} \right]}

\newtheorem{assumption}{Assumption} 
\newtheorem{conj}{Conjecture} 

\declaretheorem[name=Lemma]{lem}
\declaretheorem[name=Remark]{rem}
\declaretheorem[name=Theorem]{thm}

\definecolor{mygreen}{rgb}{0.1,0.7,0.1}
\definecolor{mydarkred}{rgb}{0.7,0.1,0}
\hypersetup{colorlinks=true,
    linkcolor=mygreen,
    citecolor=mygreen,
    filecolor=mygreen,
    urlcolor=mygreen}

\usepackage{lastpage}
\jmlrheading{22}{2021}{1-\pageref{LastPage}}{9/20; Revised
4/21}{4/21}{20-1068}{Guodong Zhang, Xuchao Bao, Laurent Lessard and Roger Grosse}
\ShortHeadings{A Unified Analysis of First-Order Methods for Smooth Games}{Zhang, Bao, Lessard, Grosse}

\begin{document}

\title{A Unified Analysis of First-Order Methods for \\ Smooth Games via Integral Quadratic Constraints}

\author{\name Guodong Zhang${}^{1, 3}$ \email gdzhang@cs.toronto.edu \\
        \name Xuchan Bao${}^{1, 3}$ \email jennybao@cs.toronto.edu \\
        \name Laurent Lessard${}^{2}$ \email l.lessard@northeastern.edu \\
        \name Roger Grosse${}^{1, 3}$ \email rgrosse@cs.toronto.edu \\
        \addr ${}^{1}$Department of Computer Science, University of Toronto \\
              ${}^{2}$Mechanical and Industrial Engineering, Northeastern University \\
              ${}^{3}$Vector Institute}

\editor{Sebastian Nowozin}

\maketitle

\begin{abstract}%
The theory of integral quadratic constraints (IQCs) allows the certification of exponential convergence of interconnected systems containing nonlinear or uncertain elements. In this work, we adapt the IQC theory to study first-order methods for smooth and strongly-monotone games and show how to design tailored quadratic constraints to get tight upper bounds of convergence rates. Using this framework, we recover the existing bound for the gradient method~(GD), derive sharper bounds for the proximal point method~(PPM) and optimistic gradient method~(OG), and provide \emph{for the first time} a global convergence rate for the negative momentum method~(NM) with an iteration complexity $\bigo(\kappa^{1.5})$, which matches its known lower bound. In addition, for time-varying systems, we prove that the gradient method with optimal step size achieves the fastest provable worst-case convergence rate with quadratic Lyapunov functions. Finally, we further extend our analysis to stochastic games and study the impact of multiplicative noise on different algorithms. We show that it is impossible for an algorithm with one step of memory to achieve acceleration if it only queries the gradient once per batch
(in contrast with the stochastic strongly-convex optimization setting, where such acceleration has been demonstrated). 
However, we exhibit an algorithm which achieves acceleration with two gradient queries per batch.
Our code is made public at \url{https://github.com/gd-zhang/IQC-Game}.
\end{abstract}%


\section{Introduction}
Gradient-based optimization algorithms have played a prominent role in machine learning and underpinned a significant fraction of the recent successes in deep learning~\citep{krizhevsky2012imagenet, silver2017mastering}. 
Typically, the training of many models can be formulated as a single-objective optimization problem, which can be efficiently solved by gradient-based optimization methods.
However, there are a growing number of models that involve multiple interacting objectives. For example, generative adversarial networks~\citep{goodfellow2014generative, radford2015unsupervised, arjovsky2017wasserstein}, adversarial training~\citep{madry2018towards} and primal-dual reinforcement learning~\citep{du2017stochastic, dai2018sbeed} all require the joint minimization of several objectives. Hence, there is a surge of interest in coupling machine learning and game theory by modeling problems as smooth games.

Smooth games, and the closely related framework of variational inequalities, are generalizations of the standard single-objective optimization framework, allowing us to model multiple players and objectives. However, new issues and challenges arise in solving smooth games or variational inequalities. Due to the conflict of optimizing different objectives, standard gradient-based algorithms may exhibit rotational behaviors~\citep{mescheder2017numerics, letcher2019differentiable} and hence converge slowly. To combat this problem, several algorithms have been introduced specifically for smooth games, including negative momentum (NM)~\citep{gidel2019negative}, optimistic gradient method (OG)~\citep{popov1980modification, rakhlin2013optimization, daskalakis2018training, mertikopoulos2018optimistic} and extra-gradient (EG)~\citep{korpelevich1976extragradient, nemirovski2004prox}. While these algorithms were motivated by provable convergence bounds, many such analyses were limited to quadratic problems with linear dynamics, or to proving local convergence (so that the dynamics could be linearized)~\citep{gidel2019negative, azizian2020accelerating, zhang2020suboptimality}. Other analyses proved global convergence rates, but relied on deep insight\footnote{Designing a Lyapunov function is largely regarded as a black art. The proofs of OG/EG are based on the insight that they both approximate the proximal point method~\citep{mokhtari2020unified}. Typically these insights do not generalize to other algorithms.} to design Lyapunov functions on a case-by-case basis~\citep{gidel2018variational, azizian2020tight, mokhtari2020unified}.

In this paper, we aim to provide a systematic framework for analyzing first-order methods in solving smooth and strongly-monotone games using techniques from control theory. 
In particular, we view common optimization algorithms as feedback interconnections and adopt the theory of integral quadratic constraints~(IQCs)~\citep{megretski1997system} to model the nonlinearities and uncertainties in the system. 
While enforcing common assumptions in optimization would seem to require infinitely many IQCs, \citet{lessard2016analysis} showed that it was possible to certify tight convergence bounds for first-order optimization algorithms using a small number of IQCs. The result of their analysis was a largely mechanical procedure for converting questions about convergence into small semidefinite programs which could be solved efficiently. We perform an analogous analysis in the more complex setting of smooth games, arriving at a very different, but similarly compact, set of IQCs. Particularly, we show that only a few pointwise IQCs are sufficient to certify tight convergence bounds for a variety of algorithms --- an even more parsimonious description than in the optimization setting.
The end result of our analysis is a unified and automated method for analyzing convergence of first-order methods for smooth games.

Using this framework, we are able to recover or even improve known convergence bounds for a variety of algorithms, which we summarize as follows:
\begin{itemize}\setlength\itemsep{0.01em}
    \item We recover the known convergence rate of the gradient method for smooth and strongly-monotone games by solving a $2 \times 2$ semidefinite program (SDP) analytically.
    \item Similarly, we derive an analytical convergence bound for the proximal point method that is sharper than the best available result~\citep[Theorem 2]{mokhtari2020unified}.
    \item We derive a slightly improved convergence rate for the optimistic gradient method (even though the existing analysis~\citep{gidel2018variational} is fairly involved).
\end{itemize}
We emphasize that all of the above results are obtainable from our unified framework through a mechanical procedure of deriving and solving an SDP.
Beyond these results, we can gain new insights and derive new results that were previously unknown and are difficult to obtain using existing approaches:
\begin{itemize}\setlength\itemsep{0.01em}
    \item We prove that, for time-varying systems, the gradient method with optimal step size achieves the fastest provable convergence rate with quadratic Lyapunov functions among any algorithm representable as a linear time-invariant system with finite state.
    \item We provide the first global convergence rate guarantee for the negative momentum method for smooth and strongly-monotone games, matching the known lower bound~\citep{zhang2020suboptimality}.
    \item We also show that the optimistic gradient method achieves the optimal convergence rate provable in our framework among algorithms with one step of memory~\eqref{eq:second-diff}. 
\end{itemize}
Further, we adapt the IQC framework to analyze stochastic games. We model stochasticity using the strong growth condition~\citep{schmidt2013fast, vaswani2019fast}, which has been used to model multiplicative noise in the optimization setting, but has not been investigated in the game setting.
The key is to model optimization algorithms as stochastic jump systems as in~\citet{hu2017unified}. We demonstrate that GD is robust to noise in the sense that it can attain the same $\bigo(\kappa^2)$ convergence rate as the deterministic case (where the constant depends on noise level).
By contrast, OG and NM are degraded to an $\bigo(\kappa^2)$ convergence rate, in contrast with their $\bigo(\kappa)$ and $\bigo(\kappa^{1.5})$ rates in the deterministic setting. We show this is an instance of a more general phenomenon: with large enough noise, no first-order algorithm with at most one step of memory can be proved under our analysis to improve upon GD's convergence rate. (This is in contrast to the setting of smooth and strongly convex optimization, where such acceleration has been proved~\citep{jain2018accelerating, vaswani2019fast}.) Nonetheless, we exhibit an algorithm which achieves acceleration in the stochastic setting by querying the vector field twice for each batch of data.

We believe our IQC framework is a powerful tool for exploratory algorithmic research, since it allows us to quickly ask and answer a variety of questions about the convergence of algorithms for smooth games.

\subsection{Other Related Works}
For the general monotone setting (without strong monotonicity), it is known that the optimal rate of convergence for first-order methods is $\bigo(1/T)$, and this rate is achieved by both the EG and OG algorithms~\citep{nemirovski2004prox, tseng2008accelerated, hsieh2019convergence, mokhtari2020convergence} for the averaged (ergodic) iterates. Later, \citep{golowich2020last, golowich2020tight} derived a $\bigo(1/\sqrt{T})$ bound for the last iterate of EG and OG.

Beyond the monotone setting, non-monotone games (e.g., nonconvex-nonconcave minimax problems) have recently gained more attention due to their generality. However, there might be no Nash (or even local Nash) equilibria in that setting due to the loss of strong duality. To overcome that, different notions of equilibrium were introduced by taking into account the sequential structure of games~\citep{jin2020local, fiez2019convergence, farnia2020gans, mangoubi2020second}. In that setting, the main challenge is to find the right equilibrium and some algorithms~\citep{wang2019solving, adolphs2019local, mazumdar2019finding} have been proposed to achieve that. 

For smooth game optimization, there are also many algorithms using high-order information. For example, consensus optimization~\citep{mescheder2017numerics}, Hamiltonian gradient descent~\citep{letcher2019differentiable, abernethy2019last}, competitive gradient descent~\citep{schafer2019competitive}, follow-the-ridge~\citep{wang2019solving} and LEAD~\citep{hemmat2020lead} all used second-order information to accelerate the convergence.
Currently, our IQC framework is primarily designed for first-order algorithms. Exploring new types of IQCs that can be used to analyze algorithms using high-order information would be an interesting future direction. 

\section{Preliminaries}\label{sec:preliminary}

\subsection{Variational Inequality Formulation of Smooth Games}
We begin by presenting the basic variational inequality framework that we consider in the sequel. Let $\Omega$ be a nonempty convex subset of $\real^d$, and let $F: \real^d \to \real^d$ be a continuous mapping on $\real^d$. In its most general form, the variational inequality (VI) problem~\citep{harker1990finite} associated to $F$ and $\Omega$ can be stated as:
\begin{equation}\label{eq:variational-inequality}
    \text{find } \bz^* \in \Omega \;\; \text{such that} \;\; F(\bz^*)^\top(\bz - \bz^*) \geq 0 \;\; \text{for all} \; \bz \in \Omega.
\end{equation}
In the case of $\Omega = \real^d$, it reduces to finding $\bz^*$ such that $F(\bz^*) = 0$. To provide some intuition about variational inequalities, we discuss two important examples below:
\begin{example}[Minimization]
Suppose that $F = \nabla_\bz f$ for a smooth function $f$ on $\real^d$, then the variational inequality problem amounts to finding the critical points of $f$. In the case where $f$ is convex, any solution of \eqref{eq:variational-inequality} is a global minimizer.
\end{example}
\begin{example}[Minimax Games]\label{examp:minimax}
Consider a convex-concave minimax optimization problem (saddle-point problem). Our objective is to solve the problem
$
    \min_\bx \max_\by f(\bx, \by)
$,
where $f$ is a smooth function. It is easy to show that minimax optimization is a special case of \eqref{eq:variational-inequality} with $F(\bz) = [\nabla_\bx f(\bx, \by)^\top, -\nabla_\by f(\bx, \by)^\top]^\top$, where $\bz = [{\bx}^\top, {\by}^\top]^\top$.
\end{example}

To be noted, the vector field $F$ in Example~\ref{examp:minimax} is not necessarily conservative, i.e., it might not be the gradient of any function. In addition, if $f$ in minimax problem is convex-concave, any solution $\bz^* = [{\bx^*}^\top, {\by^*}^\top]^\top$ of \eqref{eq:variational-inequality} is a global \emph{Nash Equilibrium}~\citep{von1944theory}:
\begin{equation}
    f(\bx^*, \by) \leq f(\bx^*, \by^*) \leq f(\bx, \by^*) \quad \text{for all } \bx \text{ and } \by \in \real^d.
\end{equation}
In this work, we are particularly interested in the case of $f$ being a strongly-convex-strongly-concave and smooth function, which basically implies that the vector field $F$ is strongly-monotone and Lipschitz (see~\citet[Lemma 2.6]{fallah2020optimal} for more details). Here we state our assumptions formally.
\begin{assumption}[Strongly Monotone]\label{ass:monotonicity}
    The vector field $F$ is $m$-strongly-monotone:
    \begin{equation}\label{eq:monotone}
        (F(\bz_1) - F(\bz_2))^\top (\bz_1 - \bz_2) \geq m \|\bz_1 - \bz_2 \|_2^2 \quad \text{for all } \bz_1, \bz_2 \in \real^d.
    \end{equation}
\end{assumption}
\begin{assumption}[Lipschitz]\label{ass:lipschitz}
    The vector field $F$ is L-Lipschitz:
    \begin{equation}\label{eq:lispchitz}
        \|F(\bz_1) - F(\bz_2) \|_2 \leq L \|\bz_1 - \bz_2 \|_2 \quad \text{for all } \bz_1, \bz_2 \in \real^d.
    \end{equation}
\end{assumption}
In the context of variational inequalites, Lipschitzness and (strong) monotonicity are fairly standard and have been used in many classical works~\citep{tseng1995linear, chen1997convergence, nesterov2007dual, nemirovski2004prox}. With these two assumptions in hand, we define the condition number $\kappa \triangleq L / m$, which measures the hardness of the problem. In the following, we turn to suitable optimization techniques for the variational inequality.

\subsection{Optimization Algorithms as Dynamical Systems}
Borrowing the notations from~\citet{lessard2016analysis}, we frame various first-order algorithms as a unified linear dynamical system\footnote{This linear dynamical system can represent any first-order methods.} in feedback with a nonlinearity $\phi: \real^d \rightarrow \real^d$,
\begin{equation}\label{eq:dynamical-system}
\begin{aligned}
    \xi_{k+1} &= A \xi_{k} + B u_k \\
    y_k &= C \xi_k + D u_k \\
    u_k &= \phi(y_k).
\end{aligned}
\end{equation}
At each iteration $k = 0, 1, ...$, $u_k \in \real^d$ is the control input, $y_k \in \real^d$ is the output, and $\xi_k \in \real^{nd}$ is the state for algorithms with $n$ step of memory. 
The state matrices $A, B, C, D$ differ for various algorithms. For most algorithms we consider in the paper, they have the general form:
\begin{equation*}
    \left[ 
    \begin{array}{c|c}
        A & B \\ \hline
        C & D
    \end{array}
    \right]
    = 
    \left[ 
    \begin{array}{cc|c}
        (1 + \beta) \iden_d & -\beta \iden_d & -\eta \iden_d \\ 
        \iden_d & \zero_d & \zero_d \\ \hline
        (1 + \alpha)\iden_d & -\alpha \iden_d & \zero_d
    \end{array}
    \right],
\end{equation*}
where $\iden_d$ and $\zero_d$ are the identity and zero matrix of size $d \times d$, respectively. One can then reduce linear dynamical system~\eqref{eq:dynamical-system} to a second-order difference equation by setting $\xi_k \coloneqq \begin{bmatrix} \bz_{k}^\top, \bz_{k-1}^\top \end{bmatrix}^\top$ and $\phi \coloneqq F$, which we term \emph{algorithms with one step of memory}:
\begin{equation}\label{eq:second-diff}
    \bz_{k+1} = (1 + \beta) \bz_k - \beta \bz_{k-1} - \eta F((1 + \alpha)\bz_k - \alpha \bz_{k-1}),
\end{equation}
where $\eta$ is a constant step size. By choosing different $\alpha, \beta$, we can recover different methods\footnote{One can also model EG using the same dynamical system \eqref{eq:dynamical-system}, but it does not fit into \eqref{eq:second-diff}.} (see Table~\ref{tab:first-order}).
For instance, optimistic gradient method (OG)~\citep{daskalakis2018training} is typically written in the following form ($\alpha=1$ and $\beta=0$).
\begin{equation}\label{eq:ogda-recursion}
    \bz_{k+1} = \bz_k - 2\eta F(\bz_k) + \eta  F(\bz_{k-1}).
\end{equation}
For smooth and strongly-monotone games, \citet[Corollary 1]{azizian2020accelerating} showed a lower bound on convergence rate for \emph{any} algorithm of the form~\eqref{eq:dynamical-system}:
\begin{equation}\label{eq:lower-bound}
    \|\xi_k - \xi^* \|_2 \geq \rho_{\text{opt}}^k \|\xi_0 - \xi^*\|_2 \quad \text{with} \quad \rho_{\text{opt}} = 1 - \frac{2m}{m + L}.
\end{equation}
Also, one can show that the lower bound for GD is $\sqrt{1 - 1/\kappa^2}$ and the lower bound for NM~\citep{zhang2020suboptimality} is $1 - c\kappa^{-1.5}$ where $c$ is a constant independent of $\kappa$.
\begin{table}[t]
    \centering
    \begin{tabular}{ |c|c|c|c| } 
     \hline
     \small{\textbf{Method}} & \small{Parameter Chioce} & \small{Complexity} & \small{Reference} \\ \hline
     \small{GD} & $\alpha = 0$, $\beta = 0$ & \small{$\bigo(\kappa^2)$} & \small{\citet{ryu2016primer, azizian2020tight}} \\ \hline
     \small{OG} & $\alpha = 1$, $\beta = 0$ & \small{$\bigo(\kappa)$} & \small{\citet{gidel2018variational, mokhtari2020unified}} \\ \hline
     \small{NM} & $\alpha = 0$, $\beta < 0$  & \small{$\bigo(\kappa^{1.5})$} & \small{Section~\ref{sec:nm} of \textbf{this paper}}\\ \hline
    \end{tabular}
    \caption{Global convergence rates of algorithms for smooth and strongly-monotone games.}
    \label{tab:first-order}
\end{table}

\subsection{IQCs for Exponential Convergence Rates}
We now present the theory of IQCs and connect it with exponential convergence. IQCs provide a convenient framework for analyzing interconnected dynamical systems that contain components that are nonlinear, uncertain, or otherwise difficult to model. The idea is to replace these troublesome components by quadratic constraints on its inputs and outputs that are known to be satisfied by all possible instances of the component. 

In our case, the vector field $F$ is the troublesome function we wish to analyze (currently, the IQC framework is limited to first-order algorithms). Although we do not know $F$ exactly, we assume to have some knowledge of the constraints it imposes on the input-output pair $(y, u)$. For example, we already assume $F$ to be $L$-Lipschitz, which implies
$\|u_k - u^*\|_2 \leq L \|y_k - y^* \|_2 $ for all $k$ with $u^* = F(y^*)$ as a fixed point. In matrix form, this is
\begin{equation}\label{eq:lipschitz-quad-constraint}
    \begin{bmatrix}
    y_k - y^* \\
    u_k - u^*
    \end{bmatrix}^\top
    \begin{bmatrix}
    L^2 \iden_d & \zero_d \\
    \zero_d & - \iden_d
    \end{bmatrix}
    \begin{bmatrix}
    y_k - y^* \\
    u_k - u^*
    \end{bmatrix} \geq 0.
\end{equation}
Notably, the above constraint is very special in that it only manifests itself as separate quadratic constraints on each $(y_k, u_k)$. It is possible to specify quadratic constraints that couple different $k$ values. To achieve that, we follow \citet{lessard2016analysis} and adopt auxiliary sequences $\zeta, s$ together with a map $\Psi$ characterized by matrices $(A_\Psi, B_\Psi^y, B_\Psi^u, C_\Psi, D_\Psi^y, D_\Psi^u)$:
\begin{equation}\label{eq:ass-ref}
\begin{aligned}
    \zeta_{k+1} &= A_\Psi \zeta_k + B_\Psi^y y_k + B_\Psi^u u_k, \\
    s_{k} &= C_\Psi \zeta_k + D_\Psi^y y_k + D_\Psi^u u_k.
\end{aligned}
\end{equation}
The equations~\eqref{eq:ass-ref} define an affine map $s = \Psi(y, u)$, where $s_k$ could be a function of all past $y_i$ and $u_i$ with $i \leq k$.
We consider the quadratic form $(s_k - s^*)^\top M (s_k-s^*)$ for a given matrix $M$ with $s^*$ and $\xi^*$ fixed points of \eqref{eq:ass-ref}. We note that the quadratic form is a function of $(y_0,\dots, y_k, u_0,\dots, u_k)$ that is determined by our choice of $(\Psi, M)$. In particular, we can recover constraint~\eqref{eq:lipschitz-quad-constraint} with
\begin{equation}\label{eq:psi}
    \Psi = \left[ 
    \begin{array}{c|cc}
        A_\Psi & B_\Psi^y & B_\Psi^u \\\hline
        C_\Psi & D_\Psi^y & D_\Psi^u \\ 
    \end{array}
    \right]
    = 
    \left[ 
    \begin{array}{c|cc}
        \zero_d & \zero_d & \zero_d \\\hline
        \zero_d & \iden_d & \zero_d \\
        \zero_d & \zero_d & \iden_d \\ 
    \end{array}
    \right], \qquad
    M = \begin{bmatrix}
    L^2 \iden_d & \zero_d \\
    \zero_d & - \iden_d
    \end{bmatrix}.
\end{equation}
In general, this sort of quadratic constraints are called IQCs. There are different types of IQCs (see~\citet[Definition 3]{lessard2016analysis}), but we will only need \emph{pointwise} IQCs as quadratic Lyapunov functions turn out to be expressive enough. 
\begin{definition}
    A \textbf{Pointwise IQC} defined by $(\Psi, M)$ satisfies
    \begin{equation*}
        (s_k - s^*)^\top M (s_k-s^*) \geq 0 \quad \text{for all } k \geq 0.
    \end{equation*}
\end{definition}
\noindent Combining the dynamics~\eqref{eq:dynamical-system} with the map $\Psi$ (by eliminating $y_k$), we obtain
\begin{equation}\label{eq:dynamics-comb}
\begin{aligned}
    \bmat{\xi_{k+1} \\ \zeta_{k+1}} &= \bmat{A & 0 \\ B_{\Psi}^y C & A_\Psi} \bmat{\xi_{k} \\ \zeta_{k}} + \bmat{B \\ B_\Psi^u + B_\Psi^y D} u_k, \\
    s_k &= \bmat{D_{\Psi}^y C & C_{\Psi}} \bmat{\xi_{k} \\ \zeta_{k}} + \bmat{D_\Psi^u + D_\Psi^y D} u_k.
\end{aligned}
\end{equation}
More succinctly, \eqref{eq:dynamics-comb} can be written as
\begin{equation}\label{eq:dynamics-compact}
\begin{aligned}
    x_{k+1} &= \hat{A} x_k + \hat{B} u_k \\
    s_k &= \hat{C}x_k + \hat{D} u_k
\end{aligned},
\qquad \text{where } x_k \triangleq \begin{bmatrix} \xi_k \\ \zeta_k \end{bmatrix}.
\end{equation}
\begin{figure}[t]
\centering
\begin{tikzpicture}[x=0.75pt,y=0.75pt,yscale=-0.9,xscale=0.9]
\draw   (100,105) -- (170,105) -- (170,145) -- (100,145) -- cycle ;
\draw   (100,173) -- (170,173) -- (170,213) -- (100,213) -- cycle ;
\draw   (170,126) -- (246,126) ;
\draw    (246,126) -- (246,196) ;
\draw[-{Latex[scale=1.0]}]    (246,196) -- (170,196) ;
\draw    (100,196) -- (41,196) -- (41,126) ;
\draw[-{Latex[scale=1.0]}]    (41,126) -- (100,126) ;
\draw   (147,127.4) .. controls (147,124.97) and (148.97,123) .. (151.4,123) -- (170,123) -- (170,145) -- (147,145) -- cycle ;
\draw[-{Latex[scale=1.0]}]   (209,196) -- (209,220) -- (260,220) ;
\draw[-{Latex[scale=1.0]}]    (71.5,196) -- (71.5,240) -- (260,240) ;
\draw   (260,210) -- (314,210) -- (314,250) -- (260,250) -- cycle ;
\draw   (291,232.4) .. controls (291,229.97) and (292.97,228) .. (295.4,228) -- (314,228) -- (314,250) -- (291,250) -- cycle ;
\draw[-{Latex[scale=1.0]}]    (314,228) -- (352,228) ;

\draw (128,120) node [anchor=north west][inner sep=0.75pt]    {$G$};
\draw (128,185) node [anchor=north west][inner sep=0.75pt]    {$\phi$};
\draw (230,160) node [anchor=north west][inner sep=0.75pt]    {$y$};
\draw (25,160) node [anchor=north west][inner sep=0.75pt]    {$u$};
\draw (153.4,126.4) node [anchor=north west][inner sep=0.75pt]  [font=\small]  {$\xi $};
\draw (275,225) node [anchor=north west][inner sep=0.75pt]    {$\Psi $};
\draw (297.4,231.4) node [anchor=north west][inner sep=0.75pt]  [font=\small]  {$\zeta $};
\draw (356,224) node [anchor=north west][inner sep=0.75pt]    {$s$};
\end{tikzpicture}
\caption{Feedback interconnection between a system $G$ (optimization algorithm) with state matrices $(A, B, C, D)$ and a nonlinearity $\phi$. An IQC is a constraint on $(y, u)$ satisfied by $\phi$ and we are mostly interested in the case where $\phi = F$.}
\label{fig:my_label}
\end{figure}
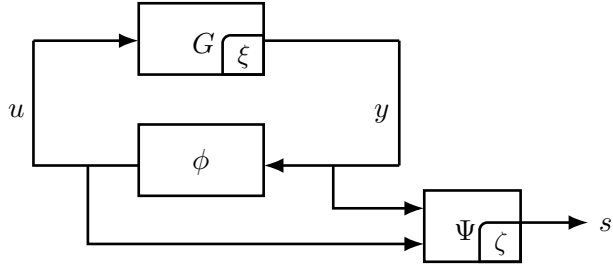

\noindent With these definitions in hand, we now state the main result of verifying exponential convergence. Basically, we build a Linear Matrix Inequality (LMI) to guide the search for the parameters of quadratic Lyapunov function in order to establish a rate bound.
\begin{restatable}{thm}{mainthm}
\label{thm:main-thm}
Consider the dynamical system~\eqref{eq:dynamical-system}.
Suppose the vector field $F$ satisfies the pointwise IQC $(\Psi, M)$ and define $(\hat A, \hat B, \hat C, \hat D)$ according to~\eqref{eq:psi}--\eqref{eq:dynamics-compact}. Consider the following linear matrix inequality (LMI):
    \begin{equation}\label{eq:sdp}
        \begin{bmatrix}
        \hat{A}^\top P \hat{A} - \rho^2 P & \hat{A}^\top P \hat{B}\\
        \hat{B}^\top P \hat{A} & \hat{B}^\top P \hat{B}
        \end{bmatrix} +
        \lambda \begin{bmatrix} \hat{C} & \hat{D} \end{bmatrix}^\top M \begin{bmatrix} \hat{C} & \hat{D} \end{bmatrix} \preceq 0.
    \end{equation}
    If this LMI is feasible for some $P \succ 0$, $\lambda \geq 0$ and $\rho > 0$\footnote{Note that $\rho$ is not necessarily smaller than $1$.}, we have
    \begin{equation}
        (x_{k+1} - x^*)^\top (P \otimes \iden_d) (x_{k+1} - x^*)  \leq \rho^2 (x_k - x^*)^\top (P \otimes \iden_d) (x_k - x^*).
    \end{equation}
    Consequently, for any $\xi_0$ and $\zeta_0 = \zeta^*$, we obtain
    \begin{equation}
        \|\xi_k - \xi^*\|_2^2 \leq \mathrm{cond}(P) \rho^{2k} \|\xi_0 - \xi^*  \|_2^2.
    \end{equation}
\end{restatable}
\begin{rem}
The LMI~\eqref{eq:sdp} can be extended to the case of multiple constraints with $(\Psi_i, M_i)$ (see~\citet[Page 12]{lessard2016analysis} for details).
\end{rem}
\begin{rem}
The positive definite quadratic function $V(x) \triangleq (x - x^*)^\top (P \otimes \iden_d) (x - x^*)$ is a Lyapunov function that certifies exponential convergence. This is the main difference from \citet[Theorem 4]{lessard2016analysis} in that the function $V(x)$ in their case cannot serve as a Lyapunov function because it does not strictly decrease over all trajectories. 
\end{rem}

To apply Theorem~\ref{thm:main-thm}, we seek to solve the semidefinite program (SDP) of finding the minimal $\rho$ such that the LMI~\eqref{eq:sdp} is feasible. For simple algorithms, one can typically solve the SDP analytically. Nevertheless, one may only get a numerical proof when the algorithm of interest is complicated and the resulting SDP is hard to solve.
The solution yields the best convergence rate that can be certified by quadratic Lyapunov functions. We remark that it \emph{automatically} searches for a quadratic Lyapunov function for proving exponetial convergence by solving the SDP. This is extremely convenient compared to designing ad-hoc Lyapunov functions on an algorithm-by-algorithm basis. Moreover, by inspecting the corresponding $\lambda_i$ of constraint $(\Psi_i, M_i)$, we could tell if the constraint or assumption is redundant or not. 
Moreover, this framework makes it easy to analyze the performance of optimization algorithms for time-varying systems, as we will show in the next section.

\section{IQCs for Variational Inequalities}\label{sec:iqc-vi}
To apply Theorem~\ref{thm:main-thm} to smooth and strongly-monotone variational inequalities, we will derive two sets of IQCs describing the vector field $F$: sector IQCs and \emph{off-by-one pointwise} IQCs. According to Assumptions~\ref{ass:monotonicity} and~\ref{ass:lipschitz}, the constraints~\eqref{eq:monotone} and \eqref{eq:lispchitz} hold over the whole domain. Hence, it has essentially infinite number of constraints and is therefore hard to use. The key idea of the following two sets of IQCs is to find the necessary conditions (but not sufficient) of smoothness and strongly-monotonicity by discretizing the constraints to finite number of $(\bz_1, \bz_2)$ pairs.
This is equivalent to a relaxation to the original problem since functions which are not $L$-Lipschitz and $m$-strongly-monotone can potentially satisfy the discretized conditions. Hence in principle, we need to make the discretized conditions to be as close to the original necessary and sufficient conditions as possible.

We first introduce two sector IQCs, which takes the discretization of $(y_k, y^*)$ with $y_k$ the output of iteration $k$ and $y^*$ the output of the stationary state.
\begin{restatable}[Sector IQCs]{lem}{sector}
\label{lem:sector-iqc}
Suppose vector field $F_k$ is $m$-strongly monotone and $L$-Lipschitz for all $k$, if $u_k = F_k(y_k)$, then $\phi := (F_0, F_1, ...)$ satisfies the \textbf{pointwise IQCs} defined by
\begin{equation}
    \Psi_1 = \Psi_2 = \left[ 
    \begin{array}{c|cc}
        \zero_d & \zero_d & \zero_d \\\hline
        \zero_d & \iden_d & \zero_d \\
        \zero_d & \zero_d & \iden_d \\ 
    \end{array}
    \right], \quad
    M_1 = \begin{bmatrix}
    L^2 \iden_d & \zero_d \\
    \zero_d & - \iden_d
    \end{bmatrix},
    \quad
    M_2 = \begin{bmatrix}
    -2 m \iden_d & \iden_d \\
    \iden_d & \zero_d
    \end{bmatrix}.
\end{equation}
We have corresponding quadratic inequalities:
\begin{equation*}
    \begin{bmatrix}
    y_k - y^* \\
    u_k - u^*
    \end{bmatrix}^\top
    \begin{bmatrix}
    L^2 \iden_d & \zero_d \\
    \zero_d & - \iden_d
    \end{bmatrix}
    \begin{bmatrix}
    y_k - y^* \\
    u_k - u^*
    \end{bmatrix} \geq 0 
\quad\text{and}\quad
    \begin{bmatrix}
    y_k - y^* \\
    u_k - u^*
    \end{bmatrix}^\top
    \begin{bmatrix}
    -2m \iden_d & \iden_d \\
    \iden_d & \zero_d
    \end{bmatrix}
    \begin{bmatrix}
    y_k - y^* \\
    u_k - u^*
    \end{bmatrix} \geq 0. 
\end{equation*}
\end{restatable}
As we will show in the next few sections, the introduced set of sector IQCs is far from sufficient for some algorithms because it allows the vector field $F_k$ to be time-varying\footnote{In convex optimization~\citep{hazan2016introduction}, it is equivalent to allowing the losses to be adversarially chosen.}, therefore leading to very conservative estimate of convergence rates. As a remedy, we can add $(y_{k-1}, y_k)$ pairs to enforce the consistency of the vector field over time, which leads to the following off-by-one pointwise IQCs. We stress that the proposed off-by-one pointwise IQC is different from the one in~\citet[Lemma 8]{lessard2016analysis} in that their off-by-one IQC is a more complicated $\rho$-hard IQC (see Definition 3 in \citet{lessard2016analysis} for details) rather than a pointwise IQC. This is due to the fact that the first-order oracle in convex minimization involves the function value $f$ and they have to use the $\rho$-hard IQC to get tight bounds.
\begin{restatable}[Off-by-one pointwise IQCs]{lem}{offbyone}\label{lem:off-by-one}
Suppose $F$ is $m$-strongly monotone and $L$-Lipschitz. If $u_k = F(y_k)$, then $\phi \coloneqq (F, F, ...)$ satisfies the \textbf{pointwise IQCs} defined by
\begin{equation}
    \Psi_1 = \Psi_2 = \left[ 
    \begin{array}{cc|c|c}
        \zero_d & \zero_d & \iden_d & \zero_d \\
        \zero_d & \zero_d & \zero_d & \iden_d \\ \hline
        -\iden_d & \zero_d & \iden_d & \zero_d \\ 
        \zero_d & -\iden_d & \zero_d & \iden_d \\ 
    \end{array}
    \right], \quad
    M_1 = \begin{bmatrix}
    L^2 \iden_d & \zero_d \\
    \zero_d & - \iden_d
    \end{bmatrix},
    \quad
    M_2 = \begin{bmatrix}
    -2 m \iden_d & \iden_d \\
    \iden_d & \zero_d
    \end{bmatrix}.
\end{equation}
We have corresponding quadratic inequalities:
\begin{equation}\label{eq:off-by-one-const}
\begin{aligned}
    L^2 \|{y}_{k+1} -  {y}_k \|_2^2 - \|{u}_{k+1} - {u}_k \|_2^2 &\geq 0, \\
    ({y}_{k+1} - {y}_{k})^\top(u_{k+1} - {u}_k - m ({y}_{k+1} -  {y}_{k})) &\geq 0.
\end{aligned}
\end{equation}
\end{restatable}
In principle, a convex combination of the sector and off-by-one pointwise IQCs is still not sufficient, though we can further add off-by-$n$ pointwise IQCs (i.e., $(y_{k-n}, y_k)$) to make it less conservative. 
To exactly characterize the nonlinearity of vector field $F$, it requires us to introduce the following interpolation condition (insipred by~\citet{taylor2017smooth}):
\begin{definition}[$\{m, L\}$-interpolation]
    Let $I$ be an index set, and consider the set of tuples $S = \{(y_i, u_i) \}_{i \in I}$. Then set $S$ is $\{m, L\}$-interpolable if and only if there exists a $m$-strongly monotone and $L$-Lipschitz vector field $F$ such that $u_i = F(y_i)$ for all $i \in I$.
\end{definition}
At first glance, it might seem that all pairs $(y_i, y_j)$ of indices $i \in I$ and $j \in I$ satisfying \eqref{eq:monotone} and~\eqref{eq:lispchitz} would be necessary and sufficient for $\{m, L\}$-interpolation. However, it was shown in~\citet[Proposition 3]{ryu2020operator} that it is not the case. In other words, including all off-by-$n$ pointwise IQCs ($n = 1, 2, ...$) is still not sufficient to fully describe the nonlinearity. 
So in general, the rate bounds certified by our IQC framework could be loose even with all off-by-$n$ IQCs and one might like to use as many IQCs as possible to make the obtained bounds less conservative.
Nevertheless,
we find in practice that it is possible to certify tight convergence bounds using only a small number of IQCs for algorithms we consider. In particular, the sector IQCs in Lemma~\ref{lem:sector-iqc} are sufficient\footnote{We mean the obtained rate bound matches the known lower bound exactly.} to provide a tight bound for GD, and adding more IQCs will not improve the rate bound obtained by our IQC framework\footnote{In particular, the corresponding $\lambda_i$ of newly added constraint $(\Psi_i, M_i)$ after solving the SDP would be zero up to numerical precision, manifesting the constraint is redundant.}. Formally, we offer the following conjecture based on our numerical simulations:
\begin{conj}\label{conj}
    For first-order algorithms with $T$ steps of memory, we only need off-by-$n$ pointwise IQCs up to $T$ to get the tightest convergence rate in our framework. In other words, adding off-by-$n$ pointwise IQCs with $n > T$ in SDP~\eqref{eq:sdp} will not improve the bound.
\end{conj}
We numerically verified our conjecture for GD and algorithms with one step of memory (see Figure~\ref{fig:conjecture})\footnote{We also tested on some algorithms with two step of memory with randomly sampled parameters.}. For instance, we notice that a combination of the sector and off-by-one pointwise IQCs is enough for algorithms with one step of memory~\eqref{eq:second-diff}. Interestingly, such combination is far from enough to get tight bounds for minimization problem.

\subsection{Warm-up: Analysis of Gradient Method}
We first warm up with the simplest algorithm -- GD. The recursion is given by
\begin{equation}
    \bz_{k+1} = \bz_k - \eta F(\bz_k).
\end{equation}
We will analyze this algorithm by applying Theorem~\ref{thm:main-thm}. The first thing is to find proper IQCs for GD. By the assumptions, we know the vector field $F$ is $m$-strongly monotone and $L$-Lipschitz. We may start with the sector IQCs defined in Lemma~\ref{lem:sector-iqc}\footnote{As we argue in Conjecture~\ref{conj}, adding more IQCs probably will not improve the bound of GD.}.

By exploiting the structure of the problem (see~\citet[section 4.2]{lessard2016analysis}), we are able to reduce the problem to the following SDP by simply setting $P = 1$ without loss of generality:
\begin{equation}\label{eq:sdp-gda}
    \begin{bmatrix}
    1 - \rho^2 & -\eta \\
    -\eta & \eta^2
    \end{bmatrix} + 
    \lambda_1 
    \begin{bmatrix}
    L^2 & 0 \\
    0 & -1
    \end{bmatrix} + 
    \lambda_2 
    \begin{bmatrix}
    -2m & 1 \\
    1 & 0
    \end{bmatrix} \preceq 0.
\end{equation}
We remark that the SDP~\eqref{eq:sdp-gda} is independent of the dimension $d$.
Using Schur complements~\citep{haynsworth1968schur}, it is equivalent to
\begin{equation}\label{eq:gda-reduction}
    \lambda_1 \geq \eta^2 \quad \lambda_2 \geq 0 \quad \rho^2 \geq 1 + \lambda_1 L^2 - 2\lambda_2 m + \frac{(\lambda_2 - \eta)^2}{\lambda_1 - \eta^2}.
\end{equation}
By analyzing the lower bound on $\rho^2$ in \eqref{eq:gda-reduction}, one can easily show that $\rho^2 \geq 1 - 2m \eta + L^2 \eta^2$. Optimizing over $\eta$, we get $\rho^2 \geq 1 - \frac{1}{\kappa^2}$, matching the lower bound of convergence rate of GD~\citep{azizian2020accelerating}. Notably, we only impose sector-bounded constraints in this section, which means the vector field can change over time (time-varying system).

After giving the warm-up example, we now present a deep result of GD for its optimality on time-varying systems.
Particularly, we show that GD with stepsize $\eta = m / L^2$ achieves the fastest possible worst-case convergence rate not only among all tunings of GD, but among any algorithm where $\bz_{k+1}$ depends linearly on $\{\bz_k, \bz_{k-1}, ..., \bz_{k-l}\}$ for some fixed $l$.
\begin{restatable}{thm}{optgd}\label{thm:opt-gda}
    With only the sector IQCs (i.e., the system can be time-varying), the best worst-case convergence rate in solving SDP~\eqref{eq:sdp} is achieved by GD with stepsize $\eta = m/L^2$ among all algorithms representable as a linear time-invariant system with finite state.
\end{restatable}
To put it differently, when the system is time-varying, we cannot improve our upper bound by using more complex algorithms. 

\subsection{Analysis of Proximal Point Method}

While GD is discretizing vector field flow with forward Euler method 
and suffers from overshotting problem, proximal point method (PPM)~\citep{rockafellar1976monotone, parikh2014proximal} adopts backward Euler method and is more stable. 
\begin{equation}
    \bz_{k+1} = \bz_k - \eta F(\bz_{k+1}).
\end{equation}
Although PPM is in general not efficiently implementable, it is largely regarded as a ``conceptual" guiding principle for accelerating optimization algorithms~\citep{drusvyatskiy2017proximal, ahn2020proximal}. 
Indeed, \citet{mokhtari2020unified} showed that both OG and EG are approximating PPM in the context of smooth games. Nevertheless, the convergence analysis of PPM is in general more involved than gradient descent method. Here we follow the same IQC pipeline to analyze its convergence rate. Notably, PPM can still be expressed as a discrete linear system as in \eqref{eq:dynamical-system} but with the matrix $D = -\eta \iden_d$. Similar to GD, we impose the sector IQCs and reduce the problem to the following SDP:
\begin{equation}\label{eq:sdp-ppm}
    \begin{bmatrix}
    1 - \rho^2 & -\eta \\
    -\eta & \eta^2
    \end{bmatrix} + 
    \lambda_1 
    \begin{bmatrix}
    L^2 & -\eta L^2 \\
    -\eta L^2 & \eta^2 L^2 -1
    \end{bmatrix} + 
    \lambda_2 
    \begin{bmatrix}
    -2m & 2\eta m + 1 \\
    2 \eta m + 1 & -2\eta^2 m - 2\eta
    \end{bmatrix} \preceq 0.
\end{equation}
Our goal is to find the minimal $\rho$ such that this LMI is feasible. To achieve that, we can use Schur complements and optimize $\lambda_1, \lambda_2$ to lower bound $\rho$. Towards this end, we are able to prove the exponential convergence of PPM by solving the LMI \eqref{eq:sdp-ppm}.
\begin{restatable}{thm}{ppm}
    Under Assumption~\ref{ass:monotonicity} and~\ref{ass:lipschitz}, PPM converges linearly with any positive $\eta$.
    \begin{equation}
        \|\bz_k - \bz^* \|_2^2 \leq \left(\frac{1}{1 + 2\eta m}\right)^k \|\bz_0 - \bz^* \|_2^2, \quad \text{for all } k \geq 0.
    \end{equation}
\end{restatable}
As opposed to the rate bound of GD, the convergence rate $\rho^2 = \tfrac{1}{1 + 2\eta m}$ does not depend on the Lipschitz constant $L$ and is strictly smaller than $1$ for all positive stepsize $\eta$. Moreover, our bound is better than the one in~\citet[Theorem 2]{mokhtari2020unified} with rate $\rho^2 = \tfrac{1}{1 + \eta m}$. It is important to know that PPM can converge arbitrarily faster with large $\eta$, but the computation of $F(\bz_{k+1})$ would become  expensive. 

\subsection{Accelerating Smooth Games With Optimism}
Optimistic gradient method (OG) was shown to be an approximation to PPM~\citep{mokhtari2020unified} and it approximates $F(\bz_{k+1})$ with a lookahead step. From this standpoint, one may expect OG to inherit the merits of PPM and potentially improve upon plain gradient method.
In this section, we analyze OG with our IQC framework and show that indeed it converges faster than GD, as also shown in~\citet{gidel2018variational}. In particular, we study the recursion of \eqref{eq:ogda-recursion}. In this case, OG is also an approximation to Extra-gradient (EG) \citep{korpelevich1976extragradient} method by using past gradient~\citep{gidel2018variational, hsieh2019convergence}:
\begin{restatable}{prop}{ogaseg}\label{prop:ogda}
    OG is an approximation to EG but using the past gradient:
    \begin{equation}\label{eq:ogaseg}
    \begin{aligned}
        \bz_{k+1/2} &= \bz_k - \eta F(\bz_{k-1/2}), \\
        \bz_{k+1} &= \bz_k - \eta F(\bz_{k+1/2}).
    \end{aligned}
    \end{equation}
\end{restatable}
\noindent Rewriting OG as a variant of EG, one can derive the following convergence result.
\begin{restatable}[{\citet{gidel2018variational}}]{thm}{ograte}\label{thm:OGDA}
    Under Assumption~\ref{ass:monotonicity} and~\ref{ass:lipschitz}, if we take $\eta = 1 / (4L)$, then
    \begin{equation}
        \|\bz_k - \bz^* \|_2^2 \leq \left(1 - \frac{1}{4 \kappa}\right)^k \|\bz_0 - \bz^* \|_2^2, \quad \text{for all } k \geq 0.
    \end{equation}
\end{restatable}
Theorem~\ref{thm:OGDA} suggests that OG has an iteration complexity of $\bigo(\kappa)$, which indeed accelerates GD substantially.
Notably, this also implies that OG is \emph{near optimal} in the sense that it matches the lower bound~\eqref{eq:lower-bound} up to a constant~\citep[Corollary 1]{azizian2020accelerating}. However, the proof of Theorem~\ref{thm:OGDA} is quite involved and relies on a cleverly designed Lyapunov function. Here, we improve the rate bound using IQC machinery. 

\begin{figure}[t]
\vspace{-0.3cm}
	\centering
    \begin{subfigure}[t]{0.45\textwidth}
        \centering
        \includegraphics[width=0.98\textwidth]{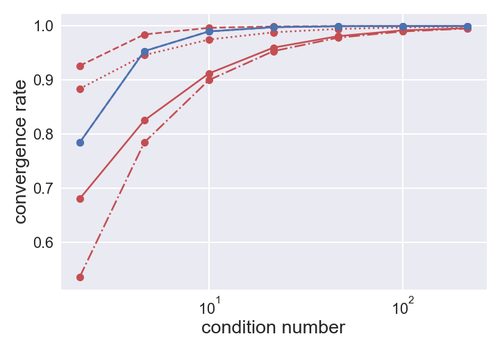}
        \vspace{-0.2cm}
        \caption{Convergence Rate}
    \end{subfigure}
    \begin{subfigure}[t]{0.45\textwidth}
        \centering
        \includegraphics[width=0.98\textwidth]{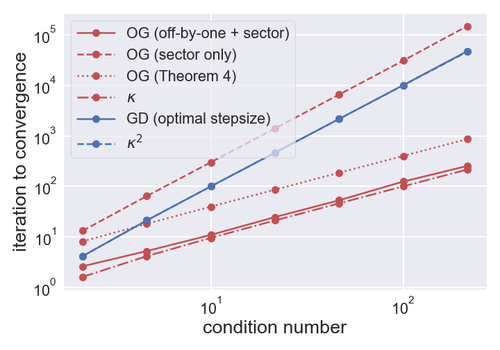}
        \vspace{-0.2cm}
        \caption{Iteration Complexity}
    \end{subfigure}
    \vspace{-0.25cm}
	\caption{Upper bounds of convergence rate and iteration complexity for GD and OG. We test both the sector IQCs and the combination of the sector and off-by-one pointwise IQCs for OG. We tuned the step sizes of OG and GD using grid search. Compared to the rate in Theorem~\ref{thm:OGDA}, we are able to improve the bound by roughly a factor of $4$. Two blue lines are virtually identical.}
	\label{fig:OGDA}
\end{figure}
We compute the rate bounds using Theorem~\ref{thm:main-thm} with either the sector IQCs in Lemma~\ref{lem:sector-iqc} or a combination of the sector and off-by-one pointwise IQCs in Lemma~\ref{lem:off-by-one}. In contrast to GD, the SDP problem induced by OG is not analytically solvable anymore, thus we use bisection search to find the optimal rate $\rho$. For fixed $\rho$ and $\kappa$, the SDP \eqref{eq:sdp} become an LMI and can be efficiently solved using interior-point methods~\citep{boyd2004convex}. For all simulations in the paper, we use CVXPY~\citep{diamond2016cvxpy} package with Mosek solver. 

With the sector IQCs alone, we observe that OG may diverge if we take $\eta = 1/4L$ in the sense that the best rate achieved is $\rho \geq 1$ even for very small condition numbers. 
To understand why, recall from Lemma~\ref{lem:sector-iqc} that the sector IQCs allow for $F_k$ to be different at each iteration. Unlike GD, OG is not robust to having a changing $F_k$. We further conjecture that the divergence of OG is caused by the aggressive step size choice (compared to $m/L^2$ in GD), we therefore tune the step size for OG. Figure~\ref{fig:OGDA} shows the certified convergence rates of OG. We find that the optimal step size for OG in that setting is much smaller than $1/4L$. Moreover, its iteration complexity scales quadratically with condition number $\kappa$ and is perhaps worse than GD by a constant. This matches the prediction of Theorem~\ref{thm:opt-gda} that GD is provably optimal for time-varying systems.

On the other hand, if we add off-by-one pointwise IQCs to the LMI, the bound for OG does improve upon that of GD, especially when the condition number is large (see red solid lines in Figure~\ref{fig:OGDA}). This suggests that enforcing the consistency of two consecutive vector field quries is important for the acceleration of OG.
In this case, the complexity of OG scales linearly with condition number, matching existing bounds.
Moreover, the convergence rate improves upon that of \citet{gidel2018variational} (see Theorem~\ref{thm:OGDA}) by roughly a constant factor of $4$, highlighting the usefulness of IQCs for certifying sharp bounds.

Lastly, we may ask the question how OG performs over the family of algorithms with one step of memory~\eqref{eq:second-diff}. This is easy to carry out numerically since one can search for the minimal $\rho^2$ for different combinations of $\alpha, \beta$. To be precise, we conduct grid search over $\alpha, \beta, \eta$ and interestingly it turns out that the optimal parameters are $\alpha = 1$ and $\beta = 0$, corresponding exactly to OG (see Figure~\ref{fig:heatmap}). In other words, OG appears likely to be optimal within the family of algorithms with one step of memory. 
\begin{figure}[t]
\vspace{-0.3cm}
	\centering
    \includegraphics[width=0.9\textwidth]{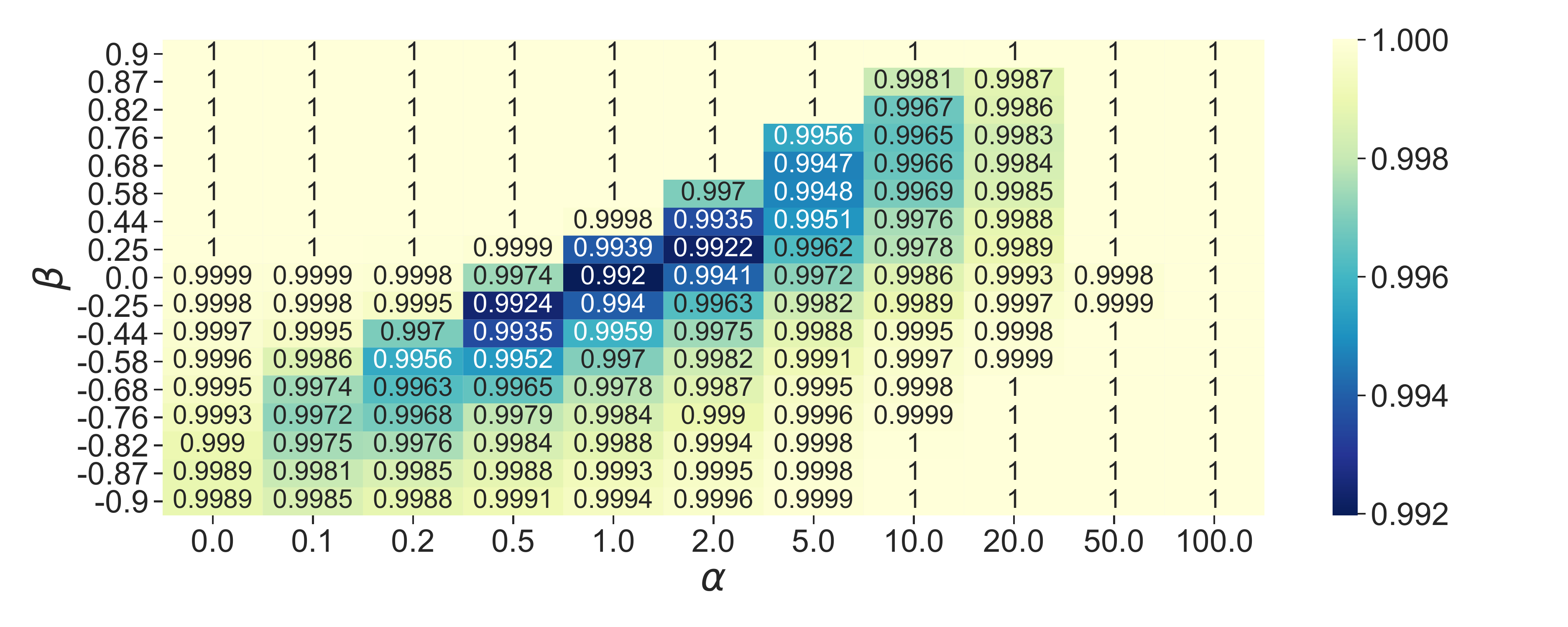}
    \vspace{-0.5cm}
	\caption{Grid search over algorithms with one step memory~\eqref{eq:second-diff}, where $\alpha$ and $\beta$ are two parameters for this algorithm family. Here, we compute the convergence rates using the condition number $\kappa = 100$. It turns out that optimistic gradient method (OG) with $\alpha = 1$ and $\beta =0$ has the best convergence rate over all combinations of $\alpha, \beta$.}
	\label{fig:heatmap}
\end{figure}

\subsection{Global Convergence of Negative Momentum}\label{sec:nm}
\begin{figure}[t]
\vspace{-0.3cm}
	\centering
    \begin{subfigure}[t]{0.45\textwidth}
        \centering
        \includegraphics[width=0.98\textwidth]{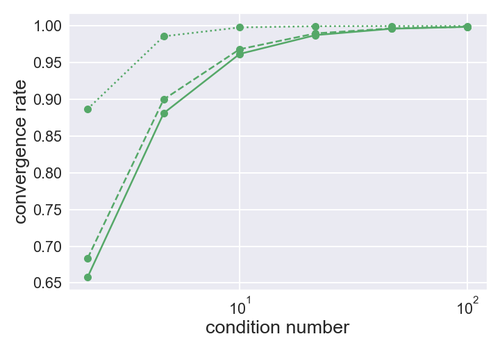}
        \vspace{-0.2cm}
        \caption{Convergence Rate}
    \end{subfigure}
    \begin{subfigure}[t]{0.45\textwidth}
        \centering
        \includegraphics[width=0.98\textwidth]{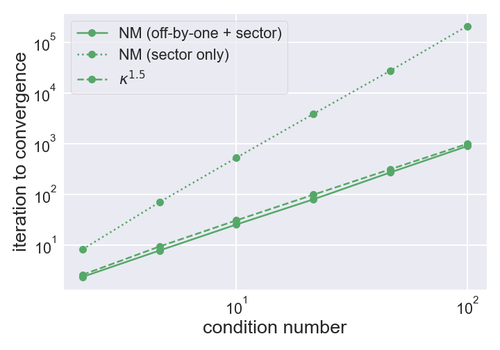}
        \vspace{-0.2cm}
        \caption{Iteration Complexity}
    \end{subfigure}
    \begin{subfigure}[t]{0.45\textwidth}
        \centering
        \includegraphics[width=0.98\textwidth]{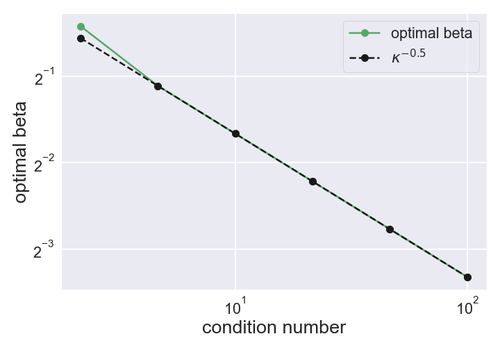}
        \vspace{-0.2cm}
        \caption{Optimal momentum parameter}\label{fig:nm-beta}
    \end{subfigure}
    \begin{subfigure}[t]{0.45\textwidth}
        \centering
        \includegraphics[width=0.98\textwidth]{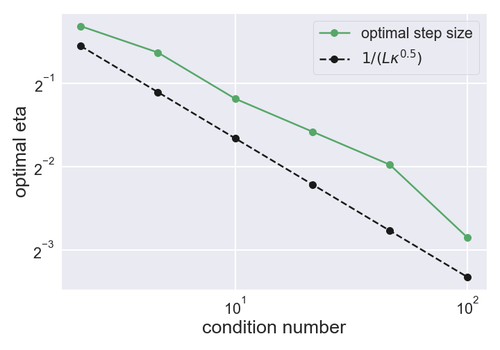}
        \vspace{-0.2cm}
        \caption{Optimal step size}\label{fig:nm-ss}
    \end{subfigure}
    \vspace{-0.25cm}
	\caption{\textbf{Top}: Curves of convergence rate and iteration complexity for negative momentum with tuned $\beta$ and $\eta$. The dashed line is the known lower bound of NM $\rho^2 = 1 - \kappa^{-1.5}$ up to a unspecified constant. The dotted line is obtained by only using the sector IQCs with $\beta = \kappa^{-0.5} - 1$. \textbf{Bottom}: Optimal momentum value and step size for negative momentum as functions of condition number $\kappa$. For momentum parameter, we plot $\beta + 1$ for clarity.}%
	\label{fig:NM}
\end{figure}
In the preceding sections, we recovered or improved previously known convergence rates of GD, PPM and OG either analytically or numerically. One may further ask whether we can provide the convergence rates of some algorithms which were unknown before based on our IQC analysis. We answer this question in the affirmative for deriving a novel convergence bound of negative momentum, which essentially refers to Polyak momentum with a negative damping parameter.

Negative momentum was first studied in~\citet{gidel2019negative} on simple bilinear games. Later, it was shown by~\citet{zhang2020suboptimality} that negative momentum converges \emph{locally} with an iteration complexity of $\bigo(\kappa^{1.5})$ for smooth and strongly-monotone variational inequality problems. More importantly, \citet{zhang2020suboptimality} showed that the bound is tight asymptotically by proving a lower bound of $\Omega(\kappa^{1.5})$. Yet, it is unclear whether negative momentum can converge globally with the same rate. In general, it is highly non-trivial to prove an explicit global convergence rate for Polyak momentum. For example, \citet{ghadimi2015global} can only show that Polyak momentum converges globally with properly chosen parameters but no explicit rate was provided.

Using a combination of the sector and off-by-one pointwise IQCs, we evaluate the rates of negative momentum numerically by doing bisection search on $\rho$ and grid search on the parameter $\beta$ and step size $\eta$. We report the results in Figure~\ref{fig:NM}. Unexpectedly, the complexity curve of negative momentum has a slope of $1.5$, suggesting it attains the same complexity of $\bigo(\kappa^{1.5})$ globally. Also, the rate matches the known lower bound tightly (see the dashed line in Figure~\ref{fig:NM}). This is surprising, in that Polyak momentum fails to achieve the same accelerated convergence rate (i.e., its local convergence rate) globally in the convex optimization setting~\citep{lessard2016analysis}. Furthermore, we find that the optimal step size and momentum parameter follow simple functions of condition number $\kappa$. According to our simulations, the optimal step size is roughly $\frac{1}{L \sqrt{\kappa}}$ while the optimal momentum value is $\kappa^{-0.5} - 1$, as shown in Figure~\ref{fig:NM}. To the best of our knowledge, we provide the \emph{first} global convergence rate guarantee for negative momentum using our IQC framework, something which is otherwise difficult to prove.

\section{IQCs for Stochastic Games}\label{sec:sto-games}
We have been discussing handling the nonlinear element $F$ of the variational inequality with IQCs. Here, we further extend it to model the uncertainty in computing the vector field $F$. Of particular interest to us is the situation where $F(\bz) = \expect_\epsilon F(\bz; \epsilon)$ is the expectation with respect to random variable $\epsilon$ of the random operator $F(\bz; \epsilon)$. Inspired by recent works on the interpolation regime, we consider a \emph{strong growth condition}~\citep{vaswani2019fast}:
\begin{equation}\label{eq:strong-growth}
    \expect_\epsilon \|F(\bz; \epsilon)\|_2^2 \leq \delta \|F(\bz)\|_2^2.
\end{equation}
Equivalently, in the finite-sum setting:
\begin{equation}
    \expect_i \|F_i(\bz)\|_2^2 \leq \delta \|F(\bz)\|_2^2.
\end{equation}
For this inequality to hold, if $F(\bz) = 0$, then $F_i(\bz) = 0$ for all $i$. This noise model is an instance of multiplicative noise in the sense that the perturbation noise is a function of the state $\bz$. Note that this noise model has been shown to hold for overparameterized models~\citep{ma2018power, liu2018mass} and underlies exponential convergence of stochastic gradient based algorithms (see~\citet{strohmer2009randomized, moulines2011non, ma2018power}). It is also possible to include additive noise by using the \emph{bias-variance} decomposition~\citep{bach2013non, fallah2020optimal}. For numerical tractability, we here focus on the finite-sum setting with $n$ examples. Later, we will show in Theorem~\ref{thm:dim-free} that the convergence rate is independent of $n$ for all $n \geq 2$. In that case, we can model optimization algorithms as stochastic jump systems~\citep{costa2006discrete}:
\begin{equation}\label{eq:jump-system}
\begin{aligned}
    \xi_{k+1} &= A \xi_{k} + B_{i_k} u_k \\
    y_k &= C \xi_k + D u_k \\
    u_k &= [F_1(y_k)^\top, ..., F_n(y_k)^\top]^\top.
\end{aligned}
\end{equation}
For the gradient method, the matrix $B_{i_k}$ is simply $(-\eta \mathbf{e}_{i_k}^\top) \otimes \iden_d $ where $\mathbf{e}_{i_k}$ is a one-hot vector with $i_k$-entry being $1$.
Similar to the deterministic system, we can impose quadratic constraints by designing $s = \Psi(y, u)$ and matrix $M$. For example, in the case of $n = 2$, we can enforce $L$-Lipschitzness of $F$ together with the strong growth condition as follows (where we ignore the dimension since we can factorize all the matrices as Kronecker products):
\begin{equation*}
    \Psi
    = 
    \left[ 
    \begin{array}{c|c|cc}
        0 & 0 & 0 & 0 \\\hline
        0 & 1 & 0 & 0 \\
        0 & 0 & 0.5 & 0.5  \\ 
        0 & 0 & 1 & 0 \\
        0 & 0 & 0 & 1 \\
    \end{array}
    \right], \quad
    M_1 = \begin{bmatrix}
    L^2 & 0 & 0 & 0\\
    0 & -1 & 0 & 0 \\
    0 & 0 & 0 & 0 \\
    0 & 0 & 0 & 0 \\
    \end{bmatrix},
    \quad
    M_2 = \begin{bmatrix}
    0 & 0 & 0 & 0 \\
    0 & 0 & 0 & 0 \\
    0 & 0 & \delta/2 - 1 & \delta/2  \\
    0 & 0 & \delta/2 & \delta/2 - 1 \\
    \end{bmatrix}.
\end{equation*}
Again, combining the dynamics~\eqref{eq:jump-system} with $\Psi$, we have the following compact form:
\begin{equation}\label{eq:jump-compact}
\begin{aligned}
    x_{k+1} &= \hat{A} x_k + \hat{B}_{i_k} u_k \\
    s_k &= \hat{C}x_k + \hat{D} u_k
\end{aligned},
\quad \text{where}\; x_k = \begin{bmatrix} \xi_k \\ \zeta_k \end{bmatrix}.
\end{equation}
Assuming $i_k$ is drawn uniformly in an i.i.d manner, we have
\begin{restatable}[{\citet[Theorem 1]{hu2017unified}}]{thm}{stojump}\label{thm:jump-system}
    Consider the stochastic jump system~\eqref{eq:jump-system}. Suppose $F$ satisfies the pointwise IQC specified by $(\Psi, M)$, and consider the following LMI:
    \begin{equation}\label{eq:sdp-jump}
        \begin{bmatrix}
        \hat{A}^\top P \hat{A} - \rho^2 P & \frac{1}{n}\sum_{i=1}^n\hat{A}^\top P \hat{B}_i \\
        \frac{1}{n}\sum_{i=1}^n\hat{B}_i^\top P \hat{A} & \frac{1}{n}\sum_{i=1}^n\hat{B}_i^\top P \hat{B}_i
        \end{bmatrix} +
        \lambda \begin{bmatrix} \hat{C} & \hat{D} \end{bmatrix}^\top M \begin{bmatrix} \hat{C} & \hat{D} \end{bmatrix} \preceq 0.
    \end{equation}
    If this LMI is feasible with $P \succ 0$ and $\lambda \geq 0$, then the following inequality holds,
    \begin{equation}
        \expect[(x_{k+1} - x^*)^\top (P \otimes \iden_d) (x_{k+1} - x^*)]  \leq \rho^2 (x_k - x^*)^\top (P \otimes \iden_d) (x_k - x^*).
    \end{equation}
    Consequently, we have
    $\expect \|\xi_k - \xi^*\|_2^2 \leq \mathrm{cond}(P) \rho^{2k}  \|\xi_0 - \xi^*  \|_2^2$ for any $\xi_0$ and $k \geq 1$.
\end{restatable}
Similar to Theorem~\ref{thm:main-thm}, when $\rho^2$ is given, the condition \eqref{eq:sdp-jump} is linear with respect to $P$ and $\lambda$. Therefore, it is an LMI whose feasible set is convex and can be effectively solved using the state-of-the-art convex optimization techniques, such as interior-point method~\citep{boyd2004convex}. 
Besides, it is important to know that the size of the LMI condition \eqref{eq:sdp-jump} scales proportionally with $n$. Nevertheless, one can show that the optimal $\rho^2$ is independent of $n$ under the strong growth condition for algorithms we consider.
\begin{restatable}{thm}{dimensionfree}\label{thm:dim-free}
    For algorithms with one step of memory, the feasible set of the LMI~\eqref{eq:sdp-jump} is independent of $n$ when $n \geq 2$, hence the optimal solution $\rho^2$ of the SDP in Theorem~\ref{thm:jump-system} under the strong growth condition is independent of $n$ when $n \geq 2$.
\end{restatable}
Therefore, we could safely choose $n = 2$ in all our numerical simulations.

\subsection{The Robustness of Gradient Method}
We now analyze the dynamics of GD to determine if it is robust to noisy gradients. 
It is easy to show (without using Theorem~\ref{thm:jump-system}) that with properly scaled step size, GD maintains the iteration complexity of $\bigo(\kappa^2)$ which we derived for the deterministic case.
\begin{restatable}[GD with the strong growth condition]{thm}{gdsto}\label{thm:SGD}
    Under Assumptions~\ref{ass:monotonicity} and~\ref{ass:lipschitz}, if we further assume the vector field $F$ satisfies the strong growth condition~\eqref{eq:strong-growth} with parameter $\delta$ and take $\eta = 1 / (L \kappa \delta)$, then we have 
    \begin{equation}
        \expect \|\bz_k - \bz^* \|_2^2 \leq \left(1 - \frac{1}{\kappa^2 \delta} \right)^k \|\bz_0 - \bz^* \|_2^2.
    \end{equation}
\end{restatable}

Compared to the rate of deterministic setting, the rate in Theorem~\ref{thm:SGD} is worse by the constant factor $\delta$. And as expected, the noisier the vector field computation (i.e., larger $\delta$), the slower the convergence. Nevertheless, the scaling with the condition number $\kappa$ matches the deterministic setting, manifesting the robustness of GD. 

\begin{figure}[t]
\vspace{-0.3cm}
	\centering
    \begin{subfigure}[t]{0.45\textwidth}
        \centering
        \includegraphics[width=0.98\textwidth]{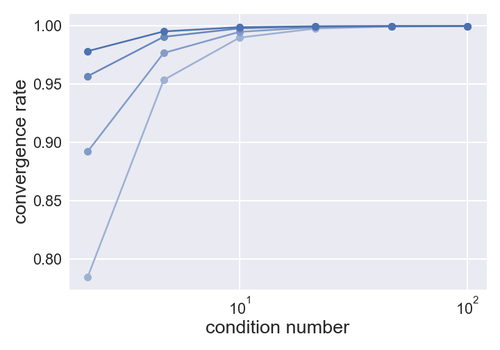}
        \vspace{-0.2cm}
        \caption{Convergence Rate}
    \end{subfigure}
    \begin{subfigure}[t]{0.45\textwidth}
        \centering
        \includegraphics[width=0.98\textwidth]{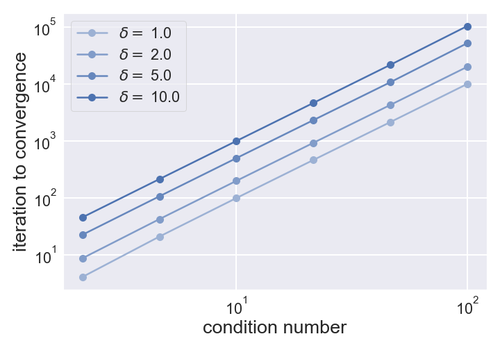}
        \vspace{-0.2cm}
        \caption{Iteration Complexity}
    \end{subfigure}
    \vspace{-0.25cm}
	\caption{Convergence rate (or iteration complexity) of GD with tuned stepsize as a function of condition number under different noise levels. $\delta = 1$ is basically the deterministic setting we studied. For a given $\delta$, it takes $\bigo(\kappa^2)$ iterations to converge no matter how large $\delta$ is.}
	\label{fig:GDA-sto}
\end{figure}
To sanity check our IQC framework, we also compute the rate bounds using Theorem~\ref{thm:jump-system} by setting $n = 2$ for convenience. We note that the result is independent of the value of $n$, choosing $n = 2$ makes the SDP problem easy to solve. We observe that the numerical rates obtained by our IQC framework match the prediction of Theorem~\ref{thm:SGD} exactly, as shown in Figure~\ref{fig:GDA-sto}. This also implies that the upper bound in Theorem~\ref{thm:SGD} is probably sharp.

\subsection{The Brittleness of Optimistic Gradient Method and Negative Momentum}
\begin{figure}[t]
\vspace{-0.3cm}
	\centering
    \begin{subfigure}[t]{0.45\textwidth}
        \centering
        \includegraphics[width=0.98\textwidth]{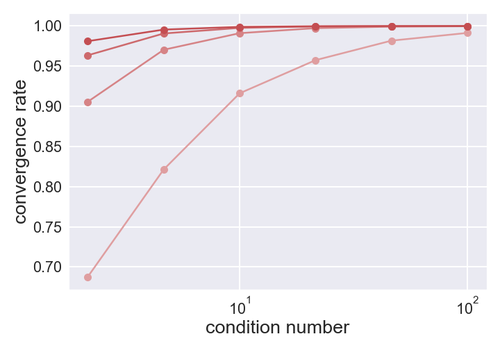}
        \vspace{-0.2cm}
        \caption{Convergence Rate}
    \end{subfigure}
    \begin{subfigure}[t]{0.45\textwidth}
        \centering
        \includegraphics[width=0.98\textwidth]{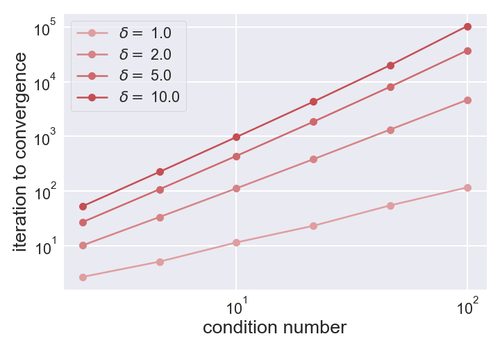}
        \vspace{-0.2cm}
        \caption{Iteration Complexity}
    \end{subfigure}
    \vspace{-0.25cm}
	\caption{Convergence rate (or iteration complexity) of OG with tuned stepsize as a function of condition number under different noise levels. In the case of $\delta = 1$ (i.e., the deterministic setting), it takes $\bigo(\kappa)$ iterations for OG to converge. Increasing the noise level with a larger $\delta$ degrades the rate. When $\delta = 10$, the iteration complexity is close to $\bigo(\kappa^2)$, which is no better than GD.}
	\label{fig:OGDA-sto}
\end{figure}
As discussed in the last section, GD is robust to multiplicative noise when it satisfies the strong growth condition~\eqref{eq:strong-growth}. It is natural to ask whether the same is true of OG and NM. Namely, are they able to match their respective deterministic convergence rates and hence accelerate GD in the stochastic setting?

We compute the convergence rates of OG using Theorem~\ref{thm:jump-system} together with the sector and off-by-one IQCs. 
We search for the optimal step size $\eta$ using grid-search. As shown in Figure~\ref{fig:OGDA-sto}, the convergence rate of optimally tuned OG deteriorates as we use $\delta > 1$ and the complexity is roughly $\bigo(\kappa^2)$ when $\delta \gg 1$. In other words, the convergence rate of OG is no better than that of GD in the stochastic setting. %

\begin{figure}[t]
\vspace{-0.3cm}
	\centering
    \begin{subfigure}[t]{0.45\textwidth}
        \centering
        \includegraphics[width=0.98\textwidth]{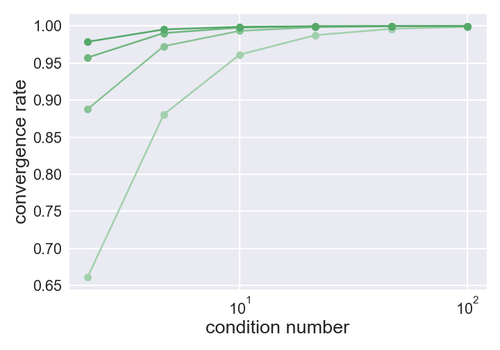}
        \vspace{-0.2cm}
        \caption{Convergence Rate}
    \end{subfigure}
    \begin{subfigure}[t]{0.45\textwidth}
        \centering
        \includegraphics[width=0.98\textwidth]{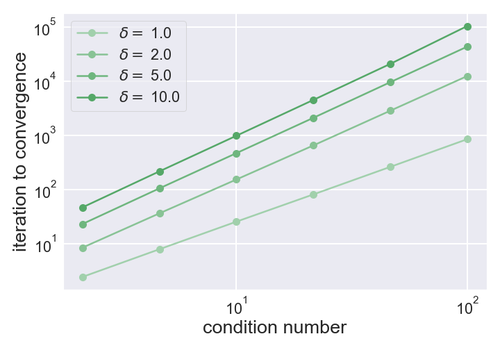}
        \vspace{-0.2cm}
        \caption{Iteration Complexity}
    \end{subfigure}
    \vspace{-0.25cm}
	\caption{Convergence rate (or iteration complexity) of NM with tuned stepsize as a function of condition number under different noise levels. In the case of $\delta = 1$ (i.e., the deterministic setting), it takes $\bigo(\kappa^{1.5})$ iterations for NM to converge. Increasing the noise level with a larger $\delta$ degrades the rate. When $\delta = 10$, the iteration complexity is close to $\bigo(\kappa^2)$, which is no better than GD.}
	\label{fig:nm-sto}
\end{figure}
We also analyze negative momentum (NM) using Theorem~\ref{thm:jump-system} with the momentum parameter $\beta = \kappa^{-0.5} - 1$ and tuned step size. Figure~\ref{fig:nm-sto} shows the plots of convergence rate and iteration complexity for different noise levels. Similar to OG, NM suffers as we gradually increase the noise level $\delta$ from $1$ to $10$. In particular, its complexity scales quadratically as a function of condition number when $\delta \gg 1$. This is to be expected by analogy with the minimization case that momentum method is fragile to injected noise.

\subsection{Is It Possible to Accelerate GD in the Stochastic Setting?}
\begin{figure}[t]
\vspace{-0.3cm}
	\centering
    \begin{subfigure}[t]{0.45\textwidth}
        \centering
        \includegraphics[width=0.98\textwidth]{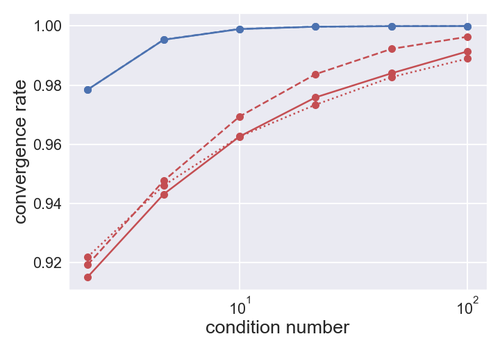}
        \vspace{-0.2cm}
        \caption{Convergence Rate}
    \end{subfigure}
    \begin{subfigure}[t]{0.45\textwidth}
        \centering
        \includegraphics[width=0.98\textwidth]{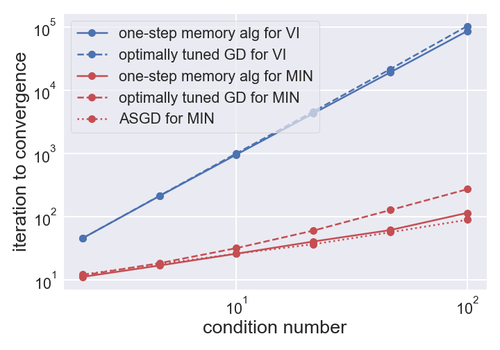}
        \vspace{-0.2cm}
        \caption{Iteration Complexity}
    \end{subfigure}
    \vspace{-0.2cm}
	\caption{Comparison between optimally tuned GD and optimally tuned one step memory algorithm in the case $\delta = 10$. For one step memory algorithm, we search $\beta + 1$ or $1 - \beta$ in log space uniformly from $[\frac{1}{\kappa \delta}, 1.0]$. For $\alpha$, we search over $[0.0, 0.1, 0.2, 0.5, 1.0, 2.0, 5.0, 10.0, 20.0, 50.0, 100.0]$. We use VI for the abbreviation of variational inequality and MIN for minimization. Accelerated Stochastic Gradient Descent (ASGD)~\citep{jain2018accelerating} is a variant of the Nesterov Accelerated Gradient which is able to accelerate SGD for minimizing strongly-convex functions under the strong growth condition. We provide a detailed proof for the acceleration effect of ASGD in Appendix~\ref{app:asgd}.}
	\label{fig:sto-comp}
\end{figure}

In the last section, we showed that both OG and NM fail to accelerate GD in the presence of noise. One may ask: \emph{does there exist any algorithm with only one step of memory that can achieve acceleration in the stochastic setting?}
In this section, we first show that acceleration is impossible if the algorithm queries each batch of data only once before moving on to the next one. We then answer our question in the affirmative by showing there exists an algorithm achieving acceleration by querying each batch of data twice.

We first search over algorithms with one step of memory~\eqref{eq:second-diff} by doing a grid search over values of $\alpha$ and $\beta$ for every particular condition number $\kappa$.
In particular, we set $\delta$ to be $10$ since the slope of resulting curve stays unchanged with larger $\delta$. This experiment is easy to carry out in our framework, because choosing new values of $\alpha$ and $\beta$ simply amounts to changing parameters in the LMI. 
We find that no algorithm is provable (under our IQC model) to obtain a faster convergence rate than the $\bigo(\kappa^2)$ rate obtained for GD (see Figure~\ref{fig:sto-comp}). This is in stark contrast to minimizing a strongly-convex function, where there is an algorithm with one step of memory accelerating GD under the strong growth condition~\citep{jain2018accelerating, vaswani2019fast}. More importantly, we do match the rate of this algorithm by conducting the same grid search over $\alpha$ and $\beta$ (see red solid line). 

\begin{figure}[t]
\vspace{-0.3cm}
	\centering
    \begin{subfigure}[t]{0.45\textwidth}
        \centering
        \includegraphics[width=0.98\textwidth]{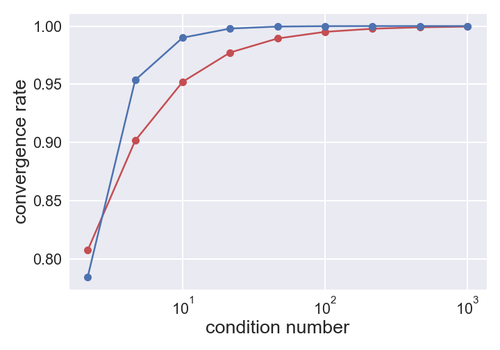}
        \vspace{-0.2cm}
        \caption{Convergence Rate}
    \end{subfigure}
    \begin{subfigure}[t]{0.45\textwidth}
        \centering
        \includegraphics[width=0.98\textwidth]{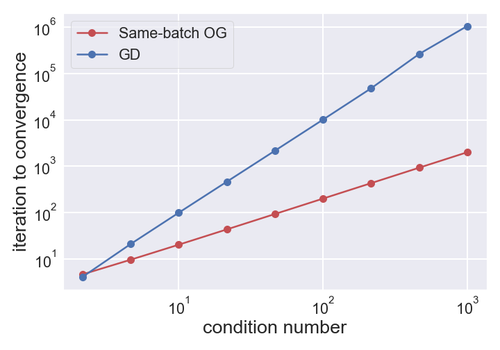}
        \vspace{-0.2cm}
        \caption{Iteration Complexity}
    \end{subfigure}
    \vspace{-0.2cm}
	\caption{Convergence rates and iteration complexities of GD and same-batch OG under the assumptions \eqref{eq:new-assumption}. By using the same stochastic operator (i.e., query the same batch twice) in \eqref{eq:same-sample-ogda}, same-batch OG accelerates GD with an iteration complexity of $\bigo(\kappa)$.}
	\label{fig:ssogda}
\end{figure}
The previous result inspires us to re-examine the reason why OG can accelerate GD in the first place when $F$ is noiseless.
By inspecting $\lambda_i$ in the final solution of the LMI~\eqref{eq:sdp}\footnote{In all four quadratic constraints, only the strongly-monotone sector IQC and the Lipschitz off-by-one pointwise IQC are used with non-zero $\lambda$.}, we notice the convergence analysis of deterministic OG heavily relies on the following property:
\begin{equation}\label{eq:ogda-lipschitz}
    \|\bz_{k+1} - \bz_{k+1/2} \| = \eta \|F(\bz_{k+1/2}) - F(\bz_{k-1/2}) \|_2 \leq \eta L \|\bz_{k+1/2} - \bz_{k-1/2} \|_2.
\end{equation}
where we used the $L$-Lipschitz assumption of $F$.
However, when the stochastic update is used, this property no longer holds. Recall the OG update in the stochastic setting:
\begin{equation}\label{eq:sto-ogda}
    \bz_{k+1} = \bz_k - \eta F(2\bz_k - \bz_{k-1}; \epsilon_k).
\end{equation}
According to Proposition~\ref{prop:ogda}, OG can be rewritten as the following form:
\begin{equation}
    \begin{aligned}
        \bz_{k+1/2} &= \bz_k - \eta F(\bz_{k-1/2}; \epsilon_{k-1}), \\
        \bz_{k+1} &= \bz_k - \eta F(\bz_{k+1/2}; \epsilon_{k}).
    \end{aligned}
\end{equation}
Observe that two different stochastic operators $F(\cdot; \epsilon_{k-1})$ and $F(\cdot ; \epsilon_{k})$ are used, so \eqref{eq:ogda-lipschitz} need not hold for any value of $L$. One can fix this problem by sharing the same stochastic operator (e.g.~using the same batch of data) to compute the updates, which we term \emph{same-batch OG}. This fix was first proposed in~\citet{mishchenko2020revisiting} for extra-gradient.
\begin{equation}\label{eq:same-sample-ogda}
    \begin{aligned}
        \bz_{k+1/2} &= \bz_k - \eta F(\bz_{k-1/2}; \epsilon_{k}), \\
        \bz_{k+1} &= \bz_k - \eta F(\bz_{k+1/2}; \epsilon_{k}).
    \end{aligned}
\end{equation}
To prove convergence, we have to replace Assumption~\eqref{ass:monotonicity} and~\eqref{ass:lipschitz} with stronger assumptions ($\bz_1, \bz_2$ could depend on $\epsilon$):
\begin{equation}\label{eq:new-assumption}
\begin{aligned}
    \expect [\|F(\bz_1; \epsilon) - F(\bz_2; \epsilon) \|_2^2] &\leq  \expect [L(\epsilon)^2 \|\bz_1 - \bz_2 \|_2^2] \leq  L^2 \expect [\|\bz_1 - \bz_2 \|_2^2] ,\\
    \expect [(F(\bz_1; \epsilon) - F(\bz_2; \epsilon))^\top (\bz_1 - \bz_2)] &\geq \expect [m(\epsilon) \|\bz_1 - \bz_2 \|_2^2] \geq m \expect [\|\bz_1 - \bz_2 \|_2^2].
\end{aligned}
\end{equation}
Basically, we allow different $F(\cdot; \epsilon)$ to have distinct Lipschitz and monotone constants. With minor modifications to our IQC analysis (see Appendix~\ref{app:sbog} for details), we can show same-batch OG~\eqref{eq:same-sample-ogda} accelerates GD with an iteration complexity of $\bigo(\kappa)$, as seen in Figure~\ref{fig:ssogda}. We remark that assumptions~\eqref{eq:new-assumption} are less restrictive than the ones used in~\citep{mishchenko2020revisiting} where they require $F(\cdot; \epsilon)$ to be almost surely strongly-monotone and Lipschitz.

\section{Discussion}
Smooth game optimization has recently emerged as a new paradigm for many models in machine learning due to its flexibility to model multiple players and their interactions. Nevertheless, the dynamics of games are more complicated than their single-objective counterparts, and raise new algorithmic challenges. 
We believe a unified and systematic analysis framework is crucial, since it could save us from the pain of analyzing algorithms in a case-by-case manner. To this end, we argue that the introduced IQC framework is a very powerful tool to study game dynamics, especially when the system contains nonlinear and uncertain ingredients. %

We note that our current framework is limited to strongly-monotone and smooth games, but other techniques from control theory (e.g., dissipativity theory~\citep{hu2017dissipativity}) may allow us to certify sublinear rates in general monotone games. 
Similarly to~\citet{lessard2016analysis}, our IQC framework could also be extended to the non-smooth setting. However, the numerical results might be less interpretable because most algorithms fail to attain linear convergence in the non-smooth setting.
Another limitation is that our IQC framework is not generally applicable for algorithms accessing higher-order information (e.g., competitive gradient descent~\citep{schafer2019competitive}). Nevertheless, for algorithms that can be written as a first-order method on modified utility functions (e.g., consensus optimization~\citep{mescheder2017numerics}\footnote{Consensus optimization can be viewed as gradient descent algorithm on modified objectives that include additional gradient norm penalties.}), it is possible to apply our IQC framework for tight convergence analysis.
Exploring new types of IQCs that can be used to analyze algorithms using high-order information would be an interesting future direction.

So far for all the algorithms we analyzed, we have shown that our IQC framework provides tight bound certification as long as the algorithm can fit into the variational inequality framework. To be noted, our framework can also be used to analyze the extra-gradient method which we did not discuss in the paper.
Nonetheless, for problems with additional structure (e.g., the bilinear saddle point problem $\min_x \max_y f(x) - g(y) + x^\top B y$), additional modifications are required to take into account the structural information for tight bounds.

Finally, one of the biggest limitations of our IQC framework is that it provides only a numerical proof, except in simpler cases where the SDP can be solved analytically. However, as noted in~\citet{lessard2016analysis}, it might be possible to find analytical proofs for complex SDPs using tools from algebraic geometry~\citep{grayson2002macaulay2, rostalski2010dualities}. Furthermore, there might exist examples that require large numbers of IQCs to get tight bounds, making the corresponding SDPs hard to solve. In our own investigations, a handful of IQCs have sufficed to obtain tight bounds, and we expect that other limited memory algorithms can be analyzed with similarly compact IQCs.

\acks
We thank Bryan Van Scoy and Yuanhao Wang for many helpful discussions. We thank Shengyang Sun, Xuechen Li and Guojun Zhang for detailed comments on early drafts. 
We also thank Adrien Taylor for pointing out a mistake of the necessary and sufficient condition for $\{m, L\}$-interpolation in the first version of our paper. Besides, we thank the anonymous JMLR reviewers for their useful feedback on earlier versions of this manuscript.

GZ would like to thank for the supports from Borealis AI fellowship and Ontario Graduate Scholarship. RG acknowledges support from the CIFAR Canadian AI Chairs program.

\newpage
\appendix

\section{Proofs for Theoretical Results}
\subsection{Proofs for Section~\ref{sec:preliminary}}
\mainthm*

\begin{proof}
    Let $x, u, s$ be a set of sequences that satisfies \eqref{eq:dynamics-compact}. Suppose $(P, \lambda)$ is a solution of SDP~\eqref{eq:sdp}. Multiply \eqref{eq:sdp} on the left and right by $[(x_k - x^*)^\top, (u_k - u^*)^\top]$ and its transpose, respectively. Making use of \eqref{eq:dynamics-compact} and \eqref{eq:ass-ref}, we obtain
    \begin{equation}\label{eq:descent-inequality}
        (x_{k+1} - x^*)^\top P(x_{k+1} - x^*) - \rho^2 (x_{k} - x^*)^\top P (x_{k} - x^*) + \lambda (s_{k} - s^*)^\top M (s_{k} - s^*) \leq 0
    \end{equation}
    Because $F$ satisfies the pointwise IQC definied by $(\Psi, M)$, therefore we obtain
    \begin{equation*}
        (x_{k+1} - x^*)^\top P(x_{k+1} - x^*) \leq \rho^2 (x_{k} - x^*)^\top P (x_{k} - x^*)
    \end{equation*}
    for all $k$ and consequently $\|x_k - x^*\|_2 \leq \sqrt{\text{cond}(P)}\rho^k \|x_0 - x^*\|_2$. Recall from \eqref{eq:dynamics-compact} that $x_k = (\xi_k, \zeta_k)$ and $\zeta_0 = \zeta^*$, we therefore have
    \begin{equation*}
    \begin{aligned}
        \|\xi_k - \xi^*\|_2^2 &\leq \|x_k - x^*\|_2^2  \\
        &\leq \text{cond}(P) \rho^{2k}\|x_0 - x^*\|_2^2 \\
        &= \text{cond}(P) \rho^{2k} (\|\xi_0 - \xi^*\|_2^2 + \|\zeta_0 - \zeta^*\|_2^2) \\
        &= \text{cond}(P) \rho^{2k} \|\xi_0 - \xi^*\|_2^2
    \end{aligned}
    \end{equation*}
    and this completes the proof.
\end{proof}

\subsection{Proofs for Section~\ref{sec:iqc-vi}}
\sector*
\begin{proof}
Two quadratic inequalities follows immediately from \eqref{eq:monotone} and \eqref{eq:lispchitz}.
\end{proof}

\offbyone*
\begin{proof}
We note two quadratic inequalities follows immediately from \eqref{eq:monotone} and \eqref{eq:lispchitz} by using $(\bz_1, \bz_2) \rightarrow (y_{k+1}, y_{k})$. To verify the IQC factorization, we note the state equations for $\Psi$ given in Lemma~\ref{lem:off-by-one} are
\begin{equation*}
    \zeta_{k+1} = \begin{bmatrix} y_k \\ u_k \end{bmatrix} \quad \text{and} \quad s_k = \begin{bmatrix} y_k - y_{k-1} \\ u_k - u_{k-1} \end{bmatrix}
\end{equation*}
and it follows that $(s_k - s^*)^\top M_1 (s_k - s^*)$ and $(s_k - s^*)^\top M_2 (s_k - s^*)$ are equivalent to quadratic constraints \eqref{eq:off-by-one-const}, as required.
\end{proof}

\optgd*
\begin{proof}
To prove the Theorem, we need to first define a family of algorithms which are expressive enough. Following on~\citet{hu2017control, lessard2019direct}, we will consider algorithms set up as in Figure~\ref{fig:iter-alg}.
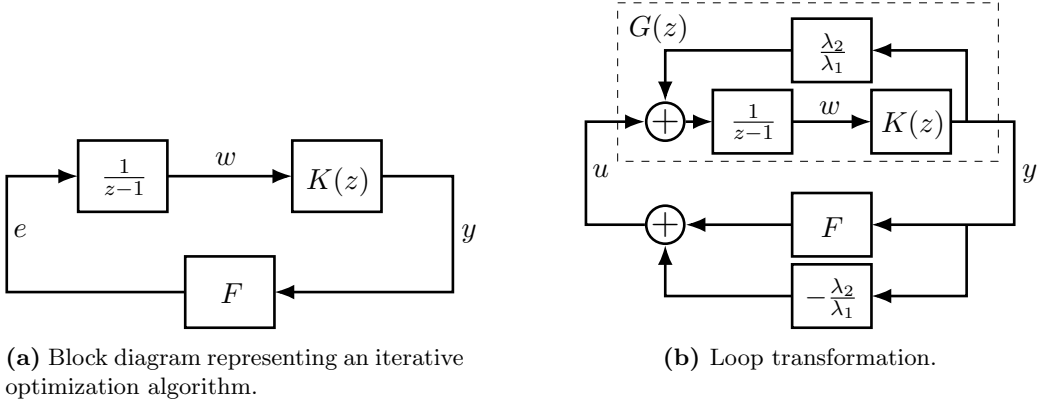
\begin{figure}[h]
\centering
\begin{subfigure}[t]{0.4\textwidth}
\centering
\begin{tikzpicture}[x=0.75pt,y=0.75pt,yscale=-0.9,xscale=0.9]
\draw   (170,105) -- (220,105) -- (220,145) -- (170,145) -- cycle ;
\draw   (110,170) -- (160,170) -- (160,210) -- (110,210) -- cycle ;
\draw   (50,105) -- (100,105) -- (100,145) -- (50,145) -- cycle ;
\draw[-{Latex[scale=1.0]}]   (220,125) -- (260,125) -- (260, 190) -- (160, 190)  ;
\draw[-{Latex[scale=1.0]}]   (110,190) -- (10,190) -- (10,125) -- (50,125) ;
\draw[-{Latex[scale=1.0]}]   (100,125) -- (170,125) ;
\draw (61,114) node [anchor=north west][inner sep=0.75pt]    {$\frac{1}{z-1}$};
\draw (176,118) node [anchor=north west][inner sep=0.75pt]    {$K(z)$};
\draw (128,183) node [anchor=north west][inner sep=0.75pt]    {$F$};
\draw (263,150) node [anchor=north west][inner sep=0.75pt]    {$y$};
\draw (12,150) node [anchor=north west][inner sep=0.75pt]    {$e$};
\draw (125,110) node [anchor=north west][inner sep=0.75pt]    {$w$};
\end{tikzpicture}
\caption{Block diagram representing an iterative optimization algorithm.}
\label{fig:iter-alg}
\end{subfigure}
\hspace{0.4cm}
\begin{subfigure}[t]{0.5\textwidth}
\centering
\begin{tikzpicture}[x=0.75pt,y=0.75pt,yscale=-0.8,xscale=0.8]
\draw (60,190) node [shape=circle,draw=black,minimum size=0.15cm][inner sep=0.pt] (sum1) {\Large$+$} ;
\draw (60,125) node [shape=circle,draw=black,minimum size=0.15cm][inner sep=0.pt] (sum2) {\Large$+$} ;
\draw   (190,105) -- (240,105) -- (240,145) -- (190,145) -- cycle ;
\draw   (140,170) -- (190,170) -- (190,210) -- (140,210) -- cycle ;
\draw   (90,105) -- (140,105) -- (140,145) -- (90,145) -- cycle ;
\draw   (140,215) -- (190,215) -- (190,255) -- (140,255) -- cycle ;
\draw   (140,60) -- (190,60) -- (190,100) -- (140,100) -- cycle ;
\draw[-{Latex[scale=1.0]}]   (240,125) -- (280,125) -- (280, 190) -- (190, 190)  ;
\draw[-{Latex[scale=1.0]}]   (140,190) -- node {} (sum1) ;
\draw[-{Latex[scale=1.0]}]   node {} (sum1) -- (10,190) -- (10,125) -- node {} (sum2) ;
\draw[-{Latex[scale=1.0]}]   node {} (sum2) -- (90,125) ;
\draw[-{Latex[scale=1.0]}]   (140,125) -- (190,125) ;
\draw[-{Latex[scale=1.0]}]   (250,190) -- (250,235) -- (190,235) ;
\draw[-{Latex[scale=1.0]}]   (140,235) -- (60, 235) -- node {} (sum1) ;
\draw[-{Latex[scale=1.0]}]   (250,125) -- (250,80) -- (190,80) ;
\draw[-{Latex[scale=1.0]}]   (140,80) -- (60, 80) -- node {} (sum2) ;
\draw (99,112) node [anchor=north west][inner sep=0.75pt] (integrator)   {$\frac{1}{z-1}$};
\draw (194,116) node [anchor=north west][inner sep=0.75pt] (K)   {$K(z)$};
\draw (156,183) node [anchor=north west][inner sep=0.75pt]    {$F$};
\draw (283,150) node [anchor=north west][inner sep=0.75pt]    {$y$};
\draw (12,150) node [anchor=north west][inner sep=0.75pt]    {$u$};
\draw (155,110) node [anchor=north west][inner sep=0.75pt]    {$w$};
\draw (154,67) node [anchor=north west][inner sep=0.75pt] (loop)    {$\frac{\lambda_2}{\lambda_1}$};
\draw (146,222) node [anchor=north west][inner sep=0.75pt]    {$-\frac{\lambda_2}{\lambda_1}$};
\draw (35,55) node [anchor=north west][inner sep=0.75pt]    {$G(z)$};

\draw[dashed,line width=0.1mm] (30,50) -- (270,50) -- (270,150) -- (30,150) -- cycle;
\end{tikzpicture}
\caption{Loop transformation.}
\label{fig:loop-trans}
\end{subfigure}
\caption{Block Diagram for control systems.}
\end{figure}

The iterative algorithm must contain a pure integrator, i.e., its transfer function must take the form $K(z) \frac{1}{z - 1}$. where $K(z)$ is an LTI system that represents the algorithm. Assume $K(z)$ has a state space representation $(A_K, B_K, C_K, D_K)$. Let $w \in \real$\footnote{Without loss of generality, we assume the whole system is single-input and single-output.} and $q \in \real^{n_K}$ be the state of integrator and $K(z)$, respectively. The order of $K(z)$, denoted $n_K$, is unspecified at this point, i.e., the algorithm may have a finite but arbitrary amount of memory. A realization of the whole algorithm then is given by
\begin{equation}
\begin{aligned}
    w_{k+1} &= w_k + e_k \\
    q_{k+1} &= A_K q_k + B_K w_k \\
    y_k &= C_K q_k + D_K w_k
\end{aligned}
\end{equation}
We remark that this family of algorithms is very general and can easily represent most algorithms. For example, we can recover momentum method by taking $K(z) = \frac{-\eta z}{z - \beta}$.

Under current assumptions about $F$, we know that we have the following quadratic constraint on the input-output pair if we only consider the sector IQCs:
\begin{equation}
    \begin{bmatrix}
    y_k - y^* \\
    e_k - e^*
    \end{bmatrix}^\top
    \begin{bmatrix}
    \lambda_1 L^2 - 2\lambda_2 m & \lambda_2 \\
    \lambda_2 &  -\lambda_1
    \end{bmatrix}
    \begin{bmatrix}
    y_k - y^* \\
    e_k - e^*
    \end{bmatrix} \geq 0 
\end{equation}
where $\lambda_1$ and $\lambda_2$ are non-negative scalars. An crucial step is now to diagonalize the quadratic constraint. In particular, we perform a loop transformation as shown in Figure~\ref{fig:loop-trans}. After the transformation, the constraint becomes
\begin{equation}
    \begin{bmatrix}
    y_k - y^* \\
    u_k - u^*
    \end{bmatrix}^\top
    \begin{bmatrix}
    \lambda_1 L^2 - 2\lambda_2 m + \frac{\lambda_2^2}{\lambda_1} & 0 \\
    0 &  -\lambda_1
    \end{bmatrix}
    \begin{bmatrix}
    y_k - y^* \\
    u_k - u^*
     \end{bmatrix} \geq 0 
\end{equation}
Notice that the input to $K(z) \frac{1}{z-1}$ is transformed in the form: $e_k = u_k + \frac{\lambda_2}{\lambda_1} y_k$. Therefore, we obtain the following state space realization of $G(z)$ in terms of $(q_k, w_k)$:
\begin{equation}\label{eq:G-K}
\left[\begin{array}{c|c}A_G & B_G \\ \hline C_G & D_G\end{array}\right]
= \left[\begin{array}{cc|c}0 & 0 & 0 \\ 0 & 1 & 1 \\ \hline
0 & 0 & 0
\end{array}\right] \\
+
\left[\begin{array}{cc}\iden & 0 \\ 0 & \frac{\lambda_2}{\lambda_1} \\ \hline 0 & 1 \end{array}\right]
\underbrace{\bmat{A_K & B_K \\ C_K & D_K}}_{K}
\left[\begin{array}{cc|c}\iden & 0 & 0 \\ 0 & 1 & 0
\end{array}\right]
\end{equation}
Combining it with the map $\Psi$ defined in Lemma~\ref{lem:sector-iqc}, we have
\begin{equation*}
    \left[ 
    \begin{array}{c|c}
        \hat{A} & \hat{B} \\ \hline \\[-1.8\medskipamount]
        \hat{C} & \hat{D}
    \end{array}
    \right]
    =
    \left[ 
    \begin{array}{c|c}
        A_G & B_G \\ \hline \\[-1.8\medskipamount]
        \bsmat{ C_G \\ 0 } & \bsmat{ D_G \\ 1 }
    \end{array}
    \right]
\end{equation*}
By Theorem~\ref{thm:main-thm}, iterates converge with rate $\rho \in (0, 1]$ if there exists $P \succ 0$ such that
\begin{equation}
    \bmat{A_G & B_G \\ C_G & D_G}^\top \bmat{ P & 0 \\ 0 & \lambda_1 L^2 - 2\lambda_2 m + \frac{\lambda_2^2}{\lambda_1}} 
    \bmat{A_G & B_G \\ C_G & D_G}
    - \bmat{\rho^2 P & 0 \\ 0 & \lambda_1} \preceq 0.
\end{equation}
According to Schur complements, we can write the equivalent condition:
\begin{equation}\label{eq:eqv-cond}
    \bmat{ \rho^2 P & 0 & A_G^\top & C_G^\top \\ 0 & \lambda_1 & B_G^\top & D_G^\top \\
A_G & B_G & P^{-1} & 0 \\ C_G & D_G & 0 & H^{-1}} \succeq 0,
\quad\text{and}\quad P \succ 0
\end{equation}
where $H \triangleq (\lambda_1 L^2 - 2\lambda_2 m + \frac{\lambda_2^2}{\lambda_1}) \geq 0$. Substituting \eqref{eq:G-K} into \eqref{eq:eqv-cond}, we obtain
\begin{equation}\label{eq:LMI}
    \underbrace{\bmat{ \rho^2 P & \bsmat{0\\0} & \bsmat{0 & 0 \\ 0 & 1} & \bsmat{0 \\ 0} \\
\bsmat{0 & 0} & \lambda_1 & \bsmat{0 & 1} & 0 \\
\bsmat{0 & 0 \\ 0 & 1} & \bsmat{0 \\ 1} & P^{-1} & \bsmat{0\\0} \\
\bsmat{0 & 0} & 0 & \bsmat{0 & 0} & H^{-1} }}_{\Theta}
+
\text{sym} \bmat{ 0 & 0 \\ 0 & 0 \\ \bsmat{\iden \\ 0} & \bsmat{ 0 \\ \tfrac{\lambda_2}{\lambda_1} } \\ 0 & 1 }K
\bmat{ \bsmat{\iden & 0 } & 0 & 0 & 0 \\ \bsmat{ 0 & 1 } & 0 & 0 & 0 } \succeq 0,
\end{equation}
where $\text{sym} X \triangleq X + X^\top$. We need to introduce a Lemma to further simplify the problem:
\begin{lem}[\citet{gahinet1994linear}]
Given a symmetric matrix $\Theta \in \real^{n\times n}$ and two matrices $P, Q$ of column dimension $n$, consider the problem of finding some matrix $\Xi$ of compatible dimensions such that
\begin{equation}\label{eq:LMI-lemma}
    \Theta + P^\top \Xi^\top Q + Q^\top \Xi P \preceq 0
\end{equation}
Denote by $W_P, W_Q$ any matrices whose columns form bases for the null spaces of $P$ and $Q$ respectively. Then there exists $\Xi$ satisfying \eqref{eq:LMI-lemma} if and only if
\begin{equation}
    W_P^\top \Theta W_P \preceq 0 \quad \text{and} \quad W_Q^\top \Theta W_Q \preceq 0
\end{equation}
\end{lem}
By this Lemma, we know that \eqref{eq:LMI} is feasible if and only if a pair of conditions hold. In that case, the conditions are:
\begin{equation}
\bmat{ \rho^2 P & \bsmat{0\\0} & \bsmat{0\\1} \\
\bsmat{ 0 & 0 } & \lambda_1 & 1 \\
\bsmat{0 & 1 } & 1  & \bsmat{0\\1}^\top \! P^{-1}\! \bsmat{0\\1} +(\tfrac{\lambda_2}{\lambda_1})^2 H^{-1} } \succeq 0, \quad
\bmat{ \lambda_1 & \bsmat{0 & 1} & 0 \\ \bsmat{ 0 \\ 1} & P^{-1} & \bsmat{0\\0} \\ 0 & \bsmat{0 & 0} & H^{-1}} \succeq 0
\end{equation}
Using Schur complements again, we have ($r \triangleq \bsmat{ 0 \\ 1}^\top  P^{-1}  \bsmat{0\\1}$)
\begin{equation}
    r + (\tfrac{\lambda_2}{\lambda_1})^2H^{-1} - (\rho^{-2}r + \tfrac{1}{\lambda_1} ) \geq 0 \quad \text{and} \quad
    r \geq \tfrac{1}{\lambda_1}
\end{equation}
After some manipulations, we have
\begin{equation}
    \rho^2 \geq 1 - 2 \tfrac{\lambda_1}{\lambda_2} m + (\tfrac{\lambda_1}{\lambda_2})^2 L^2
\end{equation}
Optimizing over $\tfrac{\lambda_1}{\lambda_2}$ yields $\rho^2 \geq 1 - \frac{1}{\kappa^2}$, which is exactly the convergence rate of GD in this setting. %
\end{proof}

\ppm*
\begin{proof}
By Theorem~\ref{thm:main-thm}, we get the following SDP by setting $P = 1$:
\begin{equation*}
    \begin{bmatrix}
    1 - \rho^2 & -\eta \\
    -\eta & \eta^2
    \end{bmatrix} + 
    \lambda_1 
    \begin{bmatrix}
    L^2 & -\eta L^2 \\
    -\eta L^2 & \eta^2 L^2 -1
    \end{bmatrix} + 
    \lambda_2 
    \begin{bmatrix}
    -2m & 2\eta m + 1 \\
    2 \eta m + 1 & -2\eta^2 m - 2\eta
    \end{bmatrix} \preceq 0
\end{equation*}
Using Schur complements, the SDP is equivalent to
\begin{equation*}
\begin{aligned}
    & \eta^2 (1 + \lambda_1 L^2 - 2\lambda_2 m) - \lambda_1 - 2\lambda_2 \eta \leq 0, \lambda_1 \geq 0, \lambda_2 \geq 0 \\
    & \rho^2 \geq 1 + \lambda_1 L^2 - 2 \lambda_2 m + \frac{[(1 + \lambda_1 L^2 - 2\lambda_2 m)\eta - \lambda_2]^2}{\lambda_1 + 2\lambda_2 \eta - \eta^2 (1 + \lambda_1 L^2 - 2\lambda_2 m)}
\end{aligned}
\end{equation*}
For notational convenience, we let $\Delta \triangleq 1 + \lambda_1 L^2 - 2\lambda_2 m$ and then we have
\begin{equation}\label{eq:rho-lower-bound}
\begin{aligned}
    & \eta^2 \Delta - \lambda_1 - 2\lambda_2 \eta \leq 0, \lambda_1 \geq 0, \lambda_2 \geq 0 \\
    & \rho^2 \geq \Delta + \frac{(\eta\Delta - \lambda_2)^2}{\lambda_1 + 2\lambda_2 \eta - \eta^2 \Delta} = \frac{\lambda_1 \Delta + \lambda_2^2}{\lambda_1 + 2\lambda_2 \eta - \eta^2 \Delta}
\end{aligned}
\end{equation}
We notice that $\rho^2$ yields the smallest value when $\eta = \lambda_2 / \Delta$. Therefore, we have
\begin{equation}\label{eq:lam2}
    \eta (1 + \lambda_1 L^2 - 2\lambda_2 m) = \lambda_2 \Rightarrow \lambda_2 = \frac{\eta (1 + \lambda_1 L^2)}{1 + 2 \eta m}
\end{equation}
Plugging \eqref{eq:lam2} back into \eqref{eq:rho-lower-bound}, we obtain
\begin{equation*}
    \rho^2 \geq \frac{\lambda_2}{\eta} = \frac{1 + \lambda_1 L^2}{1 + 2\eta m} \geq \frac{1}{1 + 2\eta m}
\end{equation*}
where the last inequality follows from the fact that $\lambda_1 \geq 0$. We finish the proof.
\end{proof}

\ogaseg*
\begin{proof}
With some manipulations, we have
\begin{equation*}
    \bz_{k+1} = \bz_k - \eta F(\bz_k - \eta F(\bz_{k-1/2}))
\end{equation*}
Notice that $\eta F(\bz_{k-1/2}) =  \bz_{k-1} - \bz_k$, we then get
\begin{equation*}
    \bz_{k+1} = \bz_k - \eta F(2\bz_k - \bz_{k-1})
\end{equation*}
We therefore conclude that OG is an approximation to EG using the past gradient.
\end{proof}

\ograte*
\begin{proof}
Recall Proposition~\ref{prop:ogda}, we have
\begin{align}\label{eq:thm4-eq1}
    \| \bz_{k+1} - \bz^* \|_2^2 &= \| \bz_{k} - \eta F(\bz_{k+1/2}) - \bz^* \|_2^2 \notag \\
    & = \|\bz_{k} - \bz^*\|_2^2 - 2 \eta F(\bz_{k+1/2})^\top (\bz_{k+1} - \bz^*) - \| \bz_{k+1} - \bz_k \|_2^2
\end{align}
Also, we notice that
\begin{align}\label{eq:thm4-eq2}
    \| \bz_{k+1} - \bz_k \|_2^2 &= \|\bz_{k+1} - \bz_{k+1/2} - \eta F(\bz_{k-1/2}) \|_2^2 \notag \\
    & = \|\bz_{k+1} - \bz_{k+1/2}\|_2^2 + \| \bz_{k+1/2} - \bz_k \|_2^2 - 2\eta F(\bz_{k-1/2})^\top(\bz_{k+1} - \bz_{k+1/2})
\end{align}
Plugging \eqref{eq:thm4-eq2} back into \eqref{eq:thm4-eq1}, we have
\begin{align}
    \| \bz_{k+1} - \bz^* \|_2^2 &= \|\bz_{k} - \bz^*\|_2^2 + \|\bz_{k+1} - \bz_{k+1/2}\|_2^2 - \| \bz_{k+1/2} - \bz_k \|_2^2 \notag \\
    &\qquad\qquad -2\eta F(\bz_{k+1/2})^\top(\bz_{k+1/2} - \bz^*) \notag \\
    & \leq \|\bz_{k} - \bz^*\|_2^2 - \| \bz_{k+1/2} - \bz_k \|_2^2 + \eta^2 L^2 \| \bz_{k+1/2} - \bz_{k-1/2} \|_2^2 \notag \\
    &\qquad\qquad -2\eta F(\bz_{k+1/2})^\top(\bz_{k+1/2} - \bz^*)
\end{align}
where we used the Lipschtiz assumption of vector field $F$. Also by strongly monotonicity, we have
\begin{align}
    -2 F(\bz_{k+1/2})^\top(\bz_{k+1/2} - \bz^*) &\leq -2m \|\bz_{k+1/2} - \bz^* \|_2^2 \notag \\
    & \leq -m \|\bz_k - \bz^* \|_2^2 + 2m \| \bz_{k+1/2} - \bz_k \|_2^2
\end{align}
We therefore have
\begin{multline}\label{eq:thm4-eq3}
    \| \bz_{k+1} - \bz^* \|_2^2 \leq (1 - \eta m)\|\bz_k - \bz^* \|_2^2 - (1 - 2\eta m)\| \bz_{k+1/2} - \bz_k \|_2^2 \\
    + \eta^2 L^2 \| \bz_{k+1/2} - \bz_{k-1/2} \|_2^2
\end{multline}
By further noticing that
\begin{align}\label{eq:thm4-eq4}
    2 \| \bz_{k+1/2} - \bz_{k-1/2} \|_2^2 &\leq 4 \| \bz_{k+1/2} - \bz_{k} \|_2^2 + 4 \| \bz_{k} - \bz_{k-1/2} \|_2^2 \notag \\
    & \leq 4 \| \bz_{k+1/2} - \bz_{k} \|_2^2 + 4 \eta^2 L^2 \| \bz_{k-1/2} - \bz_{k-3/2} \|_2^2
\end{align}
where we used the Lipschitz assumption again. Finally, combining \eqref{eq:thm4-eq3} and \eqref{eq:thm4-eq4}, we get
\begin{multline}
    \| \bz_{k+1} - \bz^* \|_2^2 \leq (1 - \eta m)\|\bz_k - \bz^* \|_2^2 - (1 - 2\eta m - 4\eta^2 L^2)\| \bz_{k+1/2} - \bz_k \|_2^2 \\
    + 4 \eta^4 L^4 \| \bz_{k-1/2} - \bz_{k-3/2} \|_2^2 - \eta^2 L^2 \| \bz_{k+1/2} - \bz_{k-1/2} \|_2^2
\end{multline}
Taking $\eta = 1/(4L)$, we have
\begin{equation*}
    \| \bz_{k+1} - \bz^* \|_2^2 + \frac{1}{16} \| \bz_{k+1/2} - \bz_{k-1/2} \|_2^2 \leq \left(1 - \frac{1}{4L} \right) \left(\| \bz_{k} - \bz^* \|_2^2 + \frac{1}{16} \| \bz_{k-1/2} - \bz_{k-3/2} \|_2^2\right)
\end{equation*}
This completes the proof.
\end{proof}

\subsection{Proofs for Section~\ref{sec:sto-games}}
\stojump*
\begin{proof}
Let $x, u, s$ be a set of sequences that satisfies \eqref{eq:jump-compact}. Take the Lynapunov function with the form $V(x_k) = (x_k - x^*)^\top P (x_k - x^*)$, we then have the following relation:
\begin{equation*}
    \expect [V(x_{k+1})] = \sum_{i=1}^n [\hat{A}(x_k - x^*) + \hat{B}_i (u_k - u^*)]^\top P [\hat{A}(x_k - x^*) + \hat{B}_i (u_k - u^*)]
\end{equation*}
Multiply \eqref{eq:sdp-jump} on the left and right by $[(x_k - x^*)^\top, (u_k - u^*)^\top]$ and its transpose, respectively. We then have for all $k$
\begin{equation*}
    \expect [V(x_{k+1})] \leq \rho^2 V(x_k)
\end{equation*}
Consequently, we have $\expect \|x_k - x^*\|_2^2 \leq \mathrm{cond}(P) \rho^{2k}  \|x_0 - x^*  \|_2^2$. Given that $\|x_k - x^*\|_2^2 = \|\xi_k - \xi^*\|_2^2 + \|\zeta_k - \zeta^*\|_2^2$, we finish the proof.
\end{proof}

\dimensionfree*
\begin{proof}
    For algorithms with one step of memory, we impose the sector and off-by-one pointwise IQCs. Along with the strong growth condition \eqref{eq:strong-growth}, it yields the following state-space matrices in \eqref{eq:jump-compact}:
    \begin{equation}
    \begin{aligned}
        & \hat{A} = \bmat{1+ \beta & -\beta & 0 & 0 \\ 1 & 0 & 0 & 0 \\ 1+\alpha & -\alpha & 0 & 0 \\ 0 & 0 & 0 & 0}, \; \hat{B}_i = \bmat{-\eta \mathbf{e}_i^\top \\ \bm{0}_{1\times n} \\ \bm{0}_{1\times n} \\ \tfrac{1}{n} \bm{1}_{1\times n}}, \; \hat{C} = \bmat{1+\alpha & -\alpha & 0 & 0 \\ 0 & 0 & 0 & 0 \\ 1+\alpha & -\alpha & -1 & 0 \\ 0 & 0 & 0 & -1 \\ \bm{0}_{n \times 1} & \bm{0}_{n \times 1} & \bm{0}_{n \times 1} & \bm{0}_{n \times 1}} \\
        & \hat{D} = \bmat{\bm{0}_{1\times n} \\ \tfrac{1}{n}\bm{1}_{1\times n} \\ \bm{0}_{1\times n} \\ \tfrac{1}{n}\bm{1}_{1\times n} \\ \iden_n},  M = \bmat{\lambda_1 L^2 - 2 \lambda_2 m & \lambda_2 & 0 & 0 & \bm{0}_{n \times 1} \\ \lambda_2 & -\lambda_1 & 0 & 0 & \bm{0}_{n \times 1} \\ 0 & 0 & \lambda_3 L^2 - 2 \lambda_4 m & \lambda_4 & \bm{0}_{n \times 1} \\ 0 & 0 & \lambda_4 & -\lambda_3 & \bm{0}_{n \times 1} \\ \bm{0}_{n \times 1} & \bm{0}_{n \times 1} & \bm{0}_{n \times 1} & \bm{0}_{n \times 1} & \tfrac{\delta \lambda_5}{n}\bm{1}_{n\times n} - \lambda_5 \iden_n}
    \end{aligned}
    \end{equation}
    By Theorem~\ref{thm:jump-system}, we have the following condition to hold:
    \begin{equation}
        \bmat{\Theta_{11} & \Theta_{12} \\ \Theta_{21} & \Theta_{22}} \triangleq \bmat{\hat{A}^\top P \hat{A} - \rho^2 P + \hat{C}^\top M \hat{C} & \frac{1}{n}\sum_{i=1}^n\hat{A}^\top P \hat{B}_i + \hat{C}^\top M \hat{D}\\
        \frac{1}{n}\sum_{i=1}^n\hat{B}_i^\top P \hat{A} + \hat{D}^\top M \hat{C} & \frac{1}{n}\sum_{i=1}^n\hat{B}_i^\top P \hat{B}_i + \hat{D}^\top M \hat{D}} \preceq 0
    \end{equation}
    By Schur complements, we have the following two equivalent conditions:
    \begin{equation}
        \Theta_{11} \preceq 0, \; \Theta_{22} - \Theta_{21} \Theta_{11}^{-1} \Theta_{12} \preceq 0
    \end{equation}
    Note that the first condition is independent of $n$, so we only need to check the second one. After some basic manipulations, we have the second condition as follows:
    \begin{equation}\label{eq:62}
        \underbrace{(\tfrac{P_{11}\eta^2}{n} - \lambda_5) \iden_n + (\tfrac{P_{44}}{n^2} - \tfrac{P_{14}\eta}{n^2} - \tfrac{P_{41}\eta}{n^2} - \tfrac{\lambda_1}{n^2} - \tfrac{\lambda_3}{n^2} +  \tfrac{\lambda_5\delta}{n}) \bm{1}_{n\times n}}_{\Theta_{22}} + \underbrace{\tfrac{K}{n^2} \bm{1}_{n \times n}}_{ \Theta_{21} \Theta_{11}^{-1} \Theta_{12}} \preceq 0,
    \end{equation}
    where $K$ is a scalar that does not depend on $n$. If $n \geq 2$, then we know that one necessary condition for \eqref{eq:62} to hold is $\tfrac{P_{11}\eta^2}{n} - \lambda_5 \leq 0$.
    Let $\lambda_5^\prime \triangleq \tfrac{\lambda_5}{n}$, we have
    \begin{equation}\label{eq:cond}
        (P_{11}\eta^2 - \lambda_5^\prime) \iden_n + (P_{44} - P_{14}\eta - P_{41}\eta - \lambda_1 - \lambda_3 + \lambda_5^\prime \delta + K) \frac{\bm{1}_{n\times n}}{n}  \preceq 0.
    \end{equation}
    To further simplify \eqref{eq:cond}, we need to introduce the following Lemma.
    \begin{lem}
        For $a\iden_n + b \tfrac{\bm{1}_{n\times n}}{n} \preceq 0$ to hold ($n \geq 2$), we have $a \leq 0$ and $a + b \leq 0$.
    \end{lem}  This Lemma can be proved immediately by showing that the matrix $a\iden_n + b \tfrac{\bm{1}_{n\times n}}{n}$ has eigenvalues $\lambda_1 = ... = \lambda_{n-1} = a$ and $\lambda_n = a + b$. One caveat is that this Lemma only hold for $n \geq 2$.
    Hence, one can show the necessary and sufficient condition of \eqref{eq:cond} is as follows.
    \begin{equation}\label{eq:cond-2}
        P_{11}\eta^2 - \lambda_5^\prime + P_{44} - P_{14}\eta - P_{41}\eta - \lambda_1 - \lambda_3 + \lambda_5^\prime \delta + K \leq 0, \; P_{11}\eta^2 - \lambda_5^\prime \leq 0.
    \end{equation}
    It is important to note that two conditions in \eqref{eq:cond-2} are all independent of the choice of $n$. Therefore, we conclude that for any choice of $P, \lambda_1, \lambda_2, \lambda_3, \lambda_4, \lambda_5^\prime$, if the LMI~\eqref{eq:sdp-jump} is feasible for a particular $n \geq 2$, then it is feasible for all $n \geq 2$. In other words, the feasible set of the LMI~\eqref{eq:sdp-jump} is invariant to the choice of $n$, which further implies the optimal $\rho$ of the corresponding SDP is independent of $n$.
    This completes the proof.
\end{proof}

\gdsto*
\begin{proof}
For any $k \geq 0$, we have
\begin{equation*}
\begin{aligned}
    \| \bz_{k+1} - \bz^*\|_2^2 &= \| \bz_{k} - \eta F(\bz_k; \epsilon_k) - \bz^*\|_2^2 \\
    &= \| \bz_{k} - \bz^*\|_2^2 - 2\eta F(\bz_k; \epsilon_k)^\top (\bz_{k} - \bz^*) + \eta^2 \|F(\bz_k; \epsilon_k) \|_2^2
\end{aligned}
\end{equation*}
We then take the expectation over $\epsilon_k$ and obtain
\begin{equation*}
\begin{aligned}
    \expect [\| \bz_{k+1} - \bz^*\|_2^2] &= \| \bz_{k} - \bz^*\|_2^2 - 2\eta \expect [F(\bz_k; \epsilon_k)]^\top (\bz_{k} - \bz^*) + \eta^2 \expect[\|F(\bz_k; \epsilon_k) \|_2^2] \\ 
    & = \| \bz_{k} - \bz^*\|_2^2 - 2\eta F(\bz_k)^\top (\bz_{k} - \bz^*) + \eta^2 \expect[\|F(\bz_k; \epsilon_k) \|_2^2] \\
    & \leq \| \bz_{k} - \bz^*\|_2^2 - 2\eta m \| \bz_{k} - \bz^*\|_2^2 + \eta^2 \expect[\|F(\bz_k; \epsilon_k) \|_2^2] \\
    & \leq (1 - 2\eta m + \eta^2 \delta L^2) \| \bz_{k} - \bz^*\|_2^2 \\
    & = \left(1 - \frac{1}{\kappa^2 \delta} \right)\| \bz_{k} - \bz^*\|_2^2
\end{aligned}
\end{equation*}
where we used the strongly-monotone assumption in the first inequality, Lipschitz assumption and the strong growth condition in the second inequality. By repeatedly taking expectation over $\epsilon_{k-1}, \epsilon_{k-2}, ...$, we conclude
\begin{equation*}
    \expect \|\bz_k - \bz^* \|_2^2 \leq \left(1 - \frac{1}{\kappa^2 \delta} \right)^k  \|\bz_0 - \bz^* \|_2^2
\end{equation*}
Hence, we prove this Theorem.
\end{proof}

\section{Additional Results}
\subsection{Evidences for Conjecture~\ref{conj}}
\begin{figure}[h]
\vspace{-0.3cm}
	\centering
    \begin{subfigure}[t]{0.45\textwidth}
        \centering
        \includegraphics[width=0.98\textwidth]{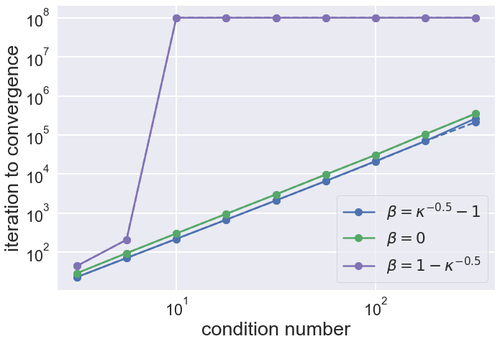}
        \vspace{-0.2cm}
        \caption{$\alpha = -1.0$}
    \end{subfigure}
    \begin{subfigure}[t]{0.45\textwidth}
        \centering
        \includegraphics[width=0.98\textwidth]{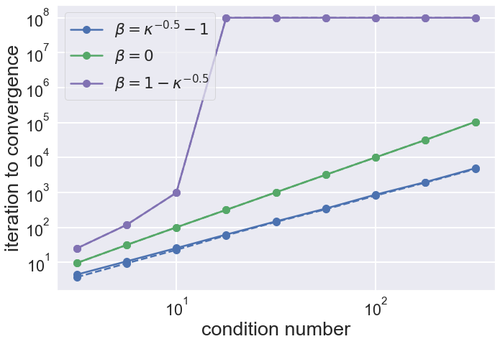}
        \vspace{-0.2cm}
        \caption{$\alpha = 0.0$}
    \end{subfigure}
    \begin{subfigure}[t]{0.45\textwidth}
        \centering
        \includegraphics[width=0.98\textwidth]{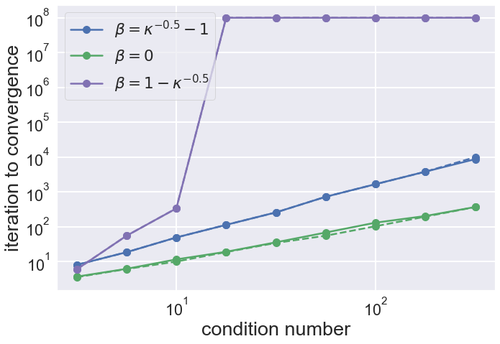}
        \vspace{-0.2cm}
        \caption{$\alpha = 1.0$}
    \end{subfigure}
    \begin{subfigure}[t]{0.45\textwidth}
        \centering
        \includegraphics[width=0.98\textwidth]{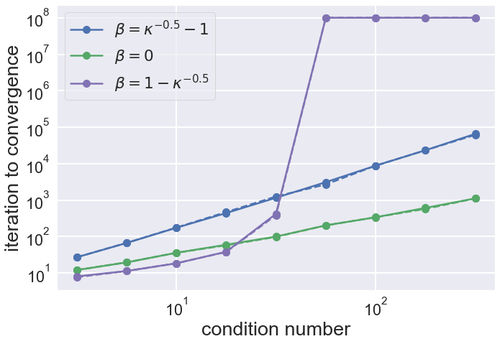}
        \vspace{-0.2cm}
        \caption{$\alpha = 5.0$}
    \end{subfigure}
    \vspace{-0.25cm}
	\caption{Curves of iteration complexity for algorithms with one step of memory~\eqref{eq:second-diff}. We pick some representative algorithms with different $\alpha, \beta$. We then do grid search over step size for every algorithm. The solid curves are the ones with only either the sector IQCs (for GD) or a combination of the sector and off-by-one pointwise IQCs while the dashed curves are obtained by adding off-by-two pointwise IQCs. Basically, we observe that off-by-two IQCs are totally redundant. In the special case of $\alpha = \beta = 0$, off-by-one IQCs are also redundant. With some particular choices of parameters, the algorithm fails to converge and we set the upper bound to be $10^8$ for visual clarity.}
	\label{fig:conjecture}
\end{figure}

\subsection{Proof of ASGD under the Strong Growth Condition}\label{app:asgd}
We note that the original ASGD~\citep{jain2018accelerating} was proposed for least squares regression, here we generalize the result to general smooth and strongly-convex functions under the strong growth condition. 
\begin{thm}[AGSD]
    Under $L$-smoothness and $m$-strong-convexity, if f satisfies the strong growth condition with constant $\delta$, then ASGD in the form of \eqref{eq:second-diff} with the following choice of parameters:
    \begin{equation}\label{eq:param-choice}
        \alpha = \frac{\sqrt{\kappa}\delta - 1}{\sqrt{\kappa} + 1}, \; \beta = \frac{\sqrt{\kappa}\delta - 1}{\sqrt{\kappa}\delta + 1}, \; \eta = \frac{\sqrt{\kappa}+1}{\sqrt{\kappa}\delta + 1}\frac{1}{L \delta}
    \end{equation}
    results in the following convergence rate:
    \begin{equation}
        \expect[f(\bz_k)] - f(\bz^*) \leq \left(1 - \frac{1}{\sqrt{\kappa}\delta} \right)^k \left(f(\bz_0) - f(\bz^*) + \frac{1}{2}m \|\bz^*\|_2^2\right)
    \end{equation}
\end{thm}
\begin{proof}
    To begin with, we rewrite ASGD in the following form:
    \begin{equation}\label{eq:comp-eq}
    \begin{aligned}
        \bx_k &= \bz_{k-1} - \eta_1 \nabla_\bz f(\bz_{k-1}; \epsilon_{k-1}) \\
        \by_k &= \bz_{k-1} - \eta_2 \nabla_\bz f(\bz_{k-1}; \epsilon_{k-1}) \\
        \bv_k & = (1 - \beta_1) \bv_{k-1} + \beta_1 \bx_k \\
        \bz_k &= (1 - \beta_2) \bv_k + \beta_2 \by_k
    \end{aligned}
    \end{equation}
    We note that the equations~\eqref{eq:comp-eq} can be compactly written as a second-order difference equation in the form of~\eqref{eq:second-diff}. We choose $\bx_0 = \by_0 = \bv_0 = 0$ for convenience. In addition, we also stress that all stationary states are the same in the sense of $\bx^* = \by^* = \bv^* = \bz^*$. Without loss of generality, we assume the minimal loss $f(\bz^*) = 0$. Now we choose the Lyapunov function of $V(k) \triangleq f(\by_k) + \tfrac{1}{2} m \|\bv_k - \bv^* \|_2^2$, then it suffice to prove
    \begin{equation}
        \expect[V(k+1)] \leq \left(1 - \frac{1}{\sqrt{\kappa}\delta} \right) V(k).
    \end{equation}
    First, we notice that 
    \begin{equation}\label{eq:first-term}
    \begin{aligned}
        \expect[f(\by_{k+1})] &= \expect[f(\bz_k - \eta_2 \nabla_\bz f(\bz_k; \epsilon_k))] \\
        & \leq f(\bz_k) - \eta_2 \|\nabla_\bz f(\bz_k)\|_2^2 + \frac{L}{2}\eta_2^2 \expect[\|\nabla_\bz f(\bz_k; \epsilon_k)\|_2^2] \\
        & \leq  f(\bz_k) - \eta_2 \|\nabla_\bz f(\bz_k)\|_2^2 + \frac{L}{2}\eta_2^2 \delta \|\nabla_\bz f(\bz_k)\|_2^2
    \end{aligned}
    \end{equation}
    where the first inequality we used the $L$-smoothness of $f$ and the second inequality we used the strong growth condition. Further, we consider the other term $\tfrac{1}{2}m \|\bv_{k+1} - \bv^*\|_2^2$.
    \begin{equation}\label{eq:velocity-term}
    \begin{aligned}
        \frac{1}{2}m \expect[\|\bv_{k+1} - \bv^* \|_2^2] & = \frac{1}{2}m \expect[\|(1 - \beta_1)(\bv_k - \bv^*) + \beta_1 (\bz_k - \bz^*) - \beta_1 \eta_1 \nabla_\bz f(\bz_k; \epsilon_k) \|_2^2] \\
        &\leq \frac{1}{2}m (1 - \beta_1) \|\bv_k - \bv^* \|_2^2 + \frac{1}{2}m \beta_1 \| \bz_k - \bz^* \|_2^2 + \frac{1}{2}m \beta_1^2\eta_1^2 \expect[\|\nabla_\bz f(\bz_k;\epsilon_k) \|_2^2] \\
        &\qquad\qquad- m \beta_1 \eta_1 ((1 - \beta_1)(\bv_k - \bv^*) + \beta_1 (\bz_k - \bz^*))^\top \nabla_\bz f(\bz_k) \\
        &\leq \frac{1}{2}m (1 - \beta_1) \|\bv_k - \bv^* \|_2^2 + \frac{1}{2}m \beta_1 \| \bz_k - \bz^* \|_2^2 + \frac{1}{2}m \beta_1^2\eta_1^2 \delta\|\nabla_\bz f(\bz_k) \|_2^2 \\
        &\qquad\qquad- m \beta_1 \eta_1 ((1 - \beta_1)(\bv_k - \bv^*) + \beta_1 (\bz_k - \bz^*))^\top \nabla_\bz f(\bz_k)
    \end{aligned}
    \end{equation}
    where we used Jensen inequality in the first inequality and then the strong growth condition in the second inequality. Next, we observe that
    \begin{align}
        (1 - \beta_1)(\bv_k - \bv^*) + \beta_1 (\bz_k - \bz^*) &= (1 - \beta_1)\frac{\bz_k - \beta_2 \by_k}{1 - \beta_2} + \beta_1 \bz_k - \bz^* \notag \\
        & = (1 - \beta_1)\frac{\beta_2}{1 - \beta_2}(\bz_k - \by_k) + \bz_k - \bz^*.
    \end{align}
    Therefore, we have
    \begin{multline}\label{eq:temp-eq}
        ((1 - \beta_1)(\bv_k - \bv^*) + \beta_1 (\bz_k - \bz^*))^\top \nabla_\bz f(\bz_k) \\
        \geq f(\bz_k) + \frac{1}{2}m \|\bz_k - \bz^*\|_2^2 
        + (1 - \beta_1)\frac{\beta_2}{1 - \beta_2} (f(\bz_k) - f(\by_k))
    \end{multline}
    where we used the $m$-strong-convexity of $f$.
    Plugging \eqref{eq:temp-eq} back into \eqref{eq:velocity-term}, we have
    \begin{multline}\label{eq:second-term}
        \frac{1}{2}m \expect[\|\bv_{k+1} - \bv^* \|_2^2]
        \leq \frac{1}{2}m (1 - \beta_1) \|\bv_k - \bv^* \|_2^2 + \frac{1}{2}m (\beta_1 - \beta_1 \eta_1 m) \| \bz_k - \bz^* \|_2^2  \\ + \frac{1}{2}m \beta_1^2\eta_1^2 \delta\|\nabla_\bz f(\bz_k) \|_2^2
        -m \beta_1 \eta_1 f(\bz_k) - m \beta_1 \eta_1 (1 - \beta_1)\frac{\beta_2}{1 - \beta_2} (f(\bz_k) - f(\by_k))
    \end{multline}
    Recall that our goal is to prove $V(k) \triangleq f(\by_k) + \tfrac{1}{2} m \|\bv_k - \bv^* \|_2^2$ is a valid Lyapunov function and also it decreases at every iteration. To achieve that, we have to choose the value of $\eta_1, \eta_2, \beta_1, \beta_2$ carefully. By inspecting \eqref{eq:first-term} and \eqref{eq:second-term}, we find the following parameters might work:
    \begin{equation}
        \eta_1 = 1/m, \quad
        \beta_1 \frac{\beta_2}{1 - \beta_2} = 1, \quad
        \frac{1}{2}m \beta_1^2 \eta_1^2 \delta - \eta_2 + \frac{L}{2}\eta_2^2 \delta \leq 0
    \end{equation}
    Combining \eqref{eq:first-term} and \eqref{eq:second-term}, we have 
    \begin{equation}
        \expect[V(k+1)] \leq (1 - \beta_1 ) V(k).
    \end{equation}
    Therefore, we could choose $\eta_1 = \frac{1}{m}$, $\eta_2 = \frac{1}{L\delta}$, $\beta_1 = \frac{1}{\sqrt{\kappa}\delta}$, $\beta_2 = \frac{\sqrt{\kappa}\delta}{\sqrt{\kappa}\delta + 1}$. Comparing \eqref{eq:comp-eq} to \eqref{eq:second-diff}, we find that this choice of parameters $\eta_1, \eta_2, \beta_1, \beta_2$ exactly corresponds to \eqref{eq:param-choice}. Hence, we finish the proof.
\end{proof}
\begin{rem}
By setting $\delta$ to be $1$ (i.e., deterministic setting), the algorithm of ASGD is exactly Nesterov's accelerated method (NAG)~\citep{nesterov1983method} and we also recover the convergence rate of NAG.
\end{rem}

\subsection{IQC Analysis for Sample-batch Optimistic Gradient Method}\label{app:sbog}
Recall the same-batch OG update in the finite-sum setting
\begin{equation}
    \begin{aligned}
        \bz_{k+1/2} &= \bz_k - \eta F_{i_k}(\bz_{k-1/2}) \\
        \bz_{k+1} &= \bz_k - \eta F_{i_k}(\bz_{k+1/2})
    \end{aligned}
\end{equation}
To model this algorithm as a discrete dynamical system, we need to take $\xi_k = [\bz_{k}^\top, \bz_{k-1/2}^\top]^\top$. We then have the state matrices as follows (in the case of $n = 2$)
\begin{equation*}
    \left[ 
    \begin{array}{c|c}
        A & B_{i_k} \\ \hline
        C & D
    \end{array}
    \right]
    = 
    \left[ 
    \begin{small}
    \begin{array}{cc|cc}
        1 & 0 & \bsmat{0 & 0} & -\eta \mathbf{e}_{i_k}^\top \\ 
        1 & 0 & -\eta \mathbf{e}_{i_k}^\top & \bsmat{0 & 0} \\ \hline
        0 & 1 & \bsmat{0 & 0} & \bsmat{0 & 0}  \\
        0 & 1 & \bsmat{0 & 0} & \bsmat{0 & 0}  \\
        1 & 0 & \bsmat{-\eta, 0} & \bsmat{0 & 0}  \\
        1 & 0 & \bsmat{0, -\eta} & \bsmat{0 & 0}
    \end{array} 
    \end{small}
    \right] \otimes \iden_d
\end{equation*}
where $\mathbf{e}_{i_k}$ is a one-hot vector with $i_k$-entry being $1$. We also have the map $\Psi$ in the following form for sector IQCs:
\begin{equation}
    \Psi_1 = \Psi_2 = \left[ 
    \begin{small}
    \begin{array}{c|cccc|cccc}
        \zero_d & \zero_d & \zero_d & \zero_d & \zero_d & \zero_d & \zero_d & \zero_d & \zero_d \\ \hline
        \zero_d & \zero_d & \zero_d & \iden_d & \zero_d & \zero_d & \zero_d & \zero_d & \zero_d \\
        \zero_d & \zero_d & \zero_d & \zero_d & \iden_d & \zero_d & \zero_d & \zero_d & \zero_d \\
        \zero_d & \zero_d & \zero_d & \zero_d & \zero_d & \zero_d & \zero_d & \iden_d & \zero_d \\
        \zero_d & \zero_d & \zero_d & \zero_d & \zero_d & \zero_d & \zero_d & \zero_d & \iden_d
    \end{array}
    \end{small}
    \right] 
\end{equation}
Similarly, for off-by-one pointwise IQCs, we have
\begin{equation}
    \Psi_3 = \Psi_4 = \left[ 
    \begin{small}
    \begin{array}{c|cccc|cccc}
        \zero_d & \zero_d & \zero_d & \zero_d & \zero_d & \zero_d & \zero_d & \zero_d & \zero_d \\ \hline
        \zero_d & -\iden_d & \zero_d & \iden_d & \zero_d & \zero_d & \zero_d & \zero_d & \zero_d \\
        \zero_d & \zero_d & -\iden_d & \zero_d & \iden_d & \zero_d & \zero_d & \zero_d & \zero_d \\
        \zero_d & \zero_d & \zero_d & \zero_d & \zero_d & -\iden_d & \zero_d & \iden_d & \zero_d \\
        \zero_d & \zero_d & \zero_d & \zero_d & \zero_d & \zero_d & -\iden_d & \zero_d & \iden_d
    \end{array}
    \end{small}
    \right] 
\end{equation}
In addition, we have the following matrices describing the assumptions:
\begin{equation*}
    M_1 = M_3 = \left[ 
    \begin{small}
    \begin{array}{cccc}
        L^2 \iden_d & \zero_d & \zero_d & \zero_d \\
        \zero_d & L^2\iden_d & \zero_d & \zero_d \\
        \zero_d & \zero_d & -\iden_d & \zero_d \\
        \zero_d & \zero_d & \zero_d & -\iden_d
    \end{array}
    \end{small}
    \right] \; \text{and} \; 
    M_2 = M_4 = \left[ 
    \begin{small}
    \begin{array}{cccc}
        -2m \iden_d & \zero_d & \iden_d & \zero_d \\
        \zero_d & -2m \iden_d & \zero_d & \iden_d \\
        \iden_d & \zero_d & \zero_d & \zero_d \\
        \zero_d & \iden_d & \zero_d & \zero_d
    \end{array}
    \end{small}
    \right]
\end{equation*}
Finally, by \eqref{eq:jump-compact} and Theorem~\ref{thm:jump-system}, we can reduce the problem to a small SDP.

\newpage
\bibliography{reference}
\end{document}